\documentclass[11pt, a4paper]{amsart}

\usepackage{amsfonts,amsmath,amssymb, amscd}

\usepackage[all]{xy}

\newtheorem{theorem}{Theorem}[section]
\newtheorem{lemma}[theorem]{Lemma}
\newtheorem{prop}[theorem]{Proposition}
\newtheorem{corollary}[theorem]{Corollary}
\newtheorem{lemmadef}[theorem]{Lemma-Definition}

\theoremstyle{definition}
\newtheorem{definition}[theorem]{Definition}
\newtheorem{rem}[theorem]{Remark}

\newtheorem{Qu}[theorem]{Question}
\newtheorem{example}[theorem]{Example}
\newtheorem{problem}[theorem]{Problem}

\usepackage{color}

\newcommand\pf{\begin{proof}}
\newcommand\epf{\end{proof}}
\newcommand{\cmdblackltimes}{\mathop{\raisebox{0.2ex}{\makebox[0.92em][l]{${\scriptstyle\blacktriangleright\mathrel{\mkern-4mu}<}$}}}}

\newcommand{\bcp}{\mathop{\raisebox{0.2ex}{\makebox[0.92em][l]{${\scriptstyle\blacktriangleright\mathrel{\mkern-4mu}\scriptstyle\vartriangleleft}$}}}}

\newcommand{\G}{\mathbb{G}}

\numberwithin{equation}{section}

\title[algebraic supergroups]
{Solvability and nilpotency for\\ algebraic supergroups}

\author[A.~Masuoka]{Akira Masuoka}
\address{Akira Masuoka,
Institute of Mathematics, 
University of Tsukuba, 
Ibaraki 305-8571, Japan}
\email{akira@math.tsukuba.ac.jp}

\author[A.~N.~Zubkov]{Alexandr N. Zubkov}
\address{Alexandr N. Zubkov, 
Omsk State Polytechnic University, Mira 11, 644050, Russia}
\email{a.zubkov@yahoo.com}

\begin{document}

\begin{abstract}
We study solvability, nilpotency and splitting property for algebraic supergroups over an arbitrary 
field $K$ of characteristic $\mathrm{char}\, K\ne 2$. Our first main theorem tells us that 
an algebraic supergroup $\mathbb{G}$ is solvable if the associated algebraic group $\mathbb{G}_{ev}$
is trigonalizable. To prove it we determine the algebraic supergroups $\mathbb{G}$ 
such that $\dim \mathrm{Lie}(\mathbb{G})_1=1$; their representations are studied when $\mathbb{G}_{ev}$
is diagonalizable. The second main theorem characterizes nilpotent connected algebraic supergroups. 
A super-analogue of the Chevalley Decomposition Theorem is proved, though it must be
in a weak form. An appendix is given to characterize smooth Noetherian superalgebras as well as
smooth Hopf superalgebras. 
\end{abstract}

\maketitle



\section{Introduction}\label{sec:introduction}

We work over an arbitrary field $K$ of characteristic $\mathrm{char}\, K \ne 2$, unless otherwise sated.
Our aim is to pursue super-analogues of the following three: 
(1)~solvability of trigonalizable algebraic groups, (2)~nilpotency criteria for connected algebraic groups, 
(3)~the Chevalley Decomposition Theorem for affine groups. 
\emph{Affine groups} (resp., \emph{algebraic groups}) mean what are called
affine group schemes (resp., algebraic affine group schemes) in \cite{water}.

General references on supersymmetry are \cite{ccf, dm, varadarajan}. 

\subsection{Basic definitions}\label{subsec:intro1}
``Super" is a synonym of ``graded by the group $\mathbb{Z}_2=\{ 0 , 1 \}$ of order 2".
Therefore, \emph{super-vector spaces} are precisely $\mathbb{Z}_2$-graded vector spaces, $V = V_0\oplus V_1$; 
the component $V_0$ (resp., $V_1$) and its elements are called \emph{even} (resp., \emph{odd}).
Those spaces form a symmetric category, $\mathsf{SMod}_K$, with respect to the so-called 
super-symmetry. The ordinary objects, such as Hopf or Lie algebras, defined in the symmetric category of
vector spaces are generalized to the super-objects, such as Hopf or Lie superalgebras, defined in 
$\mathsf{SMod}_K$. Indeed, every ordinary object, say $A$, is regarded as a 
special super-object that is \emph{purely
even} in the sense $A=A_0$. 

An \emph{affine supergroup} is a representable group-valued functor defined on 
the category $\mathsf{SAlg}_K$ of super-commutative superalgebras. Such a functor, 
say $\mathbb{G}$, is uniquely represented by a super-commutative Hopf superalgebra. 
We denote this Hopf superalgebra by $K[\mathbb{G}]$. The category $\mathsf{Alg}_K$ of commutative
algebras, or purely even super-commutative superalgebras, is a full subcategory of  
$\mathsf{SAlg}_K$. Given an affine supergroup $\mathbb{G}$, the restricted functor 
\[ \mathbb{G}_{ev} = \mathbb{G}|_{\mathsf{Alg}_K} \]
is an affine group; this is indeed represented by the Hopf algebra $K[\mathbb{G}]/(K[\mathbb{G}]_1)$, 
where $(K[\mathbb{G}]_1)$ is the Hopf super-ideal of $K[\mathbb{G}]$ generated by the odd component $K[\mathbb{G}]_1$
of $K[\mathbb{G}]$. An affine supergroup $\mathbb{G}$ is called  
an \emph{algebraic supergroup} if $K[\mathbb{G}]$ is finitely generated as an algebra. 
The associated $\mathbb{G}_{ev}$ is then an algebraic group. 

\subsection{Solvability of even-trigonalizable supergroups}\label{subsec:intro2}
An algebraic supergroup $\mathbb{G}$ is said to be \emph{even-trigonalizable} 
(resp., \emph{even-diagonalizable}) if the algebraic group $\mathbb{G}_{ev}$ is
trigonalizable (resp., diagonalizable). Our first main result, Theorem \ref{thm:even-trigonalizable},
states that every even-trigonalizable supergroup $\mathbb{G}$ has a normal chain of closed super-subgroups
\[ \mathbb{G}_0 \lhd \mathbb{G}_1\lhd \ldots\lhd \mathbb{G}_t=\mathbb{G}, \quad t \ge 0 \]
such that $\mathbb{G}_0$ is a trigonalizable algebraic group, and each factor
$\mathbb{G}_i/\mathbb{G}_{i-1}$ is isomorphic to one of the elementary supergroups
$\mathsf{G}_a^-$, $\mathsf{G}_m$ and $\mu_n$, $n>1$; see Example \ref{ex:algebraic_(super)group}. 
As a corollary, every even-trigonalizable supergroup
is solvable; this generalizes the classical result for trigonalizable algebraic groups. 
We have also that a connected smooth algebraic supergroup $\mathbb{G}$ is solvable if and only if
$\mathbb{G}_{ev}$ is solvable. These results are proved in Section \ref{sec:main_results}. A key of the proof
is to determine those algebraic supergroups $\mathbb{G}$ whose Lie superalgebra $\mathrm{Lie}(\mathbb{G})$
has one-dimensional odd component, or in notation, $\dim \mathrm{Lie}(\mathbb{G})_1=1$; this is done
in Section \ref{sec:Ggx}. Our result presents explicitly 
such an algebraic supergroup as $G_{g,x}$, parameterizing it
by an algebraic group $G$ and elements $g \in K[G]$, $x \in \mathrm{Lie}(G)$ which satisfy some conditions;
see Lemma-Definition \ref{lemdef:Ggx}. If $G$ is a diagonalizable algebraic group $D$, then
$D_{g,x}$ is even-diagonalizable. In Section \ref{sec:Dgx} 
we discuss representations of $D_{g,x}$, determining the (injective) indecomposables and the simples.
The consequences are used in Section \ref{sec:counter-examples} when we discuss (counter-)examples. 

\subsection{Nilpotency criteria for connected supergroups}\label{subsec:intro3}
These criteria are given by our second main result, Theorem \ref{thm:nilpotency}.  
In particular, it is proved that a connected algebraic supergroup $\mathbb{G}$ is nilpotent
if and only if it fits into a central extension $1 \to F \to \mathbb{G} \to \mathbb{U} \to 1$ 
of a unipotent supergroup $\mathbb{U}$ by an algebraic group $F$ of multiplicative type. 
Section \ref{sec:nilpotent_supergroup} is devoted to proving the theorem. 
An ingredient is Proposition \ref{prop:center_ev}, 
which describes the algebraic group $\mathcal{Z}(\mathbb{G})_{ev}$ 
associated with the center $\mathcal{Z}(\mathbb{G})$ of an algebraic supergroup $\mathbb{G}$.

\subsection{Harish-Chandra pairs}\label{subsec:intro4}
To prove our results so far stated, a crucial role will be
played by the category equivalence between the algebraic supergroups and the Harish-Chandra pairs.
A \emph{Harish-Chandra pair} is a pair $(G,V)$ of an algebraic group $G$ and a finite-dimensional
right $G$-module, given as a structure a $G$-equivariant 
bilinear map $V \times V \to \mathrm{Lie}(G)$ satisfying some conditions; see Definition \ref{def:HCP}. 
This concept is due to Kostant \cite{kostant}. The category equivalence was proved by Koszul \cite{koszul} 
in the $C^{\infty}$ situation, and by Vishnyakova \cite{vish} in the complex analytic situation; see also
\cite[Section 7.4]{ccf}.  
In the algebraic situation as cited above, it was proved by
the first-named author \cite{mas3}, 
who generalized, in purely Hopf-algebraic terms, the result 
by Carmeli and Fioresi \cite{carfi} proved when $K=\mathbb{C}$; see also \cite{masshiba}. 
The result is reproduced as Theorem \ref{thm:equivalence}
in a suitable form for our argument. Section \ref{sec:HCP} is devoted also to reproducing from \cite{mas3}  
some needed results on Harish-Chandra pairs. 
The results allow us to study algebraic supergroups
through Harish-Chandra pairs which are much easier to handle with,
and produce many results in this paper, as they already did in \cite{mas3, zubgrish}. The section, indeed,
contains a new application; see Section \ref{subsec:characterize_super-diagonalizable}. 

\subsection{Super-analogue of the Chevalley Decomposition Theorem}\label{subsec:intro5}
Given an affine supergroup $\mathbb{G}$, we let $\mathbb{G}_u$ denote its unipotent radical. 
In the last two Sections \ref{sec:counter-examples} and \ref{sec:splitting_property} we ask the following question
for some sorts of $\mathbb{G}$. 
\begin{itemize}
\item[(Q)] Does the quotient morphism $\mathbb{G} \to \mathbb{G}/\mathbb{G}_u$ split?
\end{itemize}
If yes, we have $\mathbb{G} \simeq \mathbb{G}/\mathbb{G}_u \ltimes \mathbb{G}_u$. 

Proposition \ref{prop:split} answers (Q) in the positive under the assumptions 
(i)~$\mathrm{char}\, K=0$,
and (ii)~$\mathbb{G}/\mathbb{G}_u$ is linearly reductive. 
This is a super-analogue of the Chevalley Decomposition Theorem, but it, assuming (ii), is weaker in 
that for affine groups, (ii) is proved to hold under (i); the assumption (ii) is indeed necessary
in the super context; see Remark \ref{rem:split}. 
The proposition also gives a positive answer,
as a direct consequence of a classical result on affine groups, when (i)~$K$ is an algebraically closed
field of $\mathrm{char}\, K>2$, and (ii)~$\mathbb{G}/\mathbb{G}_u$ is a diagonalizable affine group. 

In Section \ref{subsec:counter-example1} we answer (Q) in the negative, giving a counter-example, when 
$\mathbb{G}/\mathbb{G}_u$ is even-diagonalizable. 
In Section \ref{subsec:counter-example2} we answer
another question, Question \ref{qu:super-diagonalizable}, in the negative; a counter-example
shows that even-diagonalizable algebraic supergroups $\mathbb{G}$ 
are not so simple even if $\mathbb{G}_u=1$. 
\vspace{0mm}

The text starts with the following Section \ref{sec:preliminaries} which is devoted to preliminaries
on Hopf superalgebras and supergroups.

An appendix has been added to answer a question posed to an earlier version of this paper,
which concerns definitions of smoothness.
The answer is given by Proposition \ref{prop:compare_with_Fioresi}, 
which shows the equivalence of our definition
with Definition 3.1 of Fioresi \cite{fioresi}. 
More essential are Theorem \ref{thm:smooth}
and Proposition \ref{prop:smooth_Hopf}, which
characterize smooth Noetherian superalgebras and smooth Hopf superalgebras, respectively;
they would be new, and of independent interest. 
The paper \cite{schmitt} by Schmitt is crucial
for the appendix. 

\section{Preliminaries}\label{sec:preliminaries}

\subsection{}\label{subsec:Hopf_superalgebra} 

In this paper, Hopf superalgebras play an important role. Given a Hopf superalgebra $A$, the coproduct, 
the counit and the antipode will be denoted by $\Delta$, $\varepsilon$ and $\mathcal{S}$, respectively, or by
$\Delta_A$, $\varepsilon_A$ and $\mathcal{S}_A$, respectively. For the coproduct we will use the
Sweedler notation of the form
\[ \Delta(a)=a_{(1)}\otimes a_{(2)}, \quad 
\Delta(a_{(1)})\otimes a_{(2)}=
a_{(1)}\otimes a_{(2)}\otimes a_{(3)}. \]

A non-zero element $g \in A$ is said to be 
\emph{grouplike} if it is even, and $\Delta(g) = g \otimes g$; we have then necessarily 
$\varepsilon(g) = 1$, $\mathcal{S}(g)=g^{-1}$. We use thus the word in a more restricted sense 
than usual, except
in Remark \ref{rem:inhomogeneous_grouplike}, which shows existence of inhomogeneous ``grouplikes". 
An element $x \in A$ is said to be \emph{primitive} if $\Delta(x) = 1 \otimes x + x \otimes 1$, and 
necessarily, $\varepsilon(x) = 0$, $\mathcal{S}(x)=-x$. All primitives in $A$ form a
super-vector subspace of $A$, which we denote by $P(A)$, whence $P(A) =P(A)_0\oplus P(A)_1$. 
An element $x \in A$ is said to be \emph{skew-primitive} if $\Delta(x) = h \otimes x + x \otimes g$
for some grouplikes $g, h \in A$, and  
necessarily, $\varepsilon(x) = 0$, $\mathcal{S}(x)=-h^{-1}xg^{-1}$. 

We let $A^+=\mathrm{Ker}\, \varepsilon$ denote the augmentation ideal of $A$. 

\subsection{}\label{subsec:supergroups}
We emphasize that
every affine group, say $G$, is regarded as the affine supergroup which assigns to each
$R\in \mathsf{SAlg}_K$, the group $G(R_0)$ of points in the even component $R_0$ of $R$. 
Therefore, given an affine supergroup $\mathbb{G}$, the associated affine group $\mathbb{G}_{ev}$
is regarded as the closed super-subgroup of $\mathbb{G}$ such that $\mathbb{G}_{ev}(R) = \mathbb{G}(R_0)$, 
$R\in \mathsf{SAlg}_K$. 

We say that an algebraic supergroup $\mathbb{G}$ is \emph{connected}, if the largest   
(purely even, commutative) finite-dimensional separable subalgebra $\pi_0(K[\mathbb{G}])$ 
of $K[\mathbb{G}]$ is trivial.
This is equivalent to saying that the associated algebraic group $\mathbb{G}_{ev}$ is connected
\cite[Page 51]{water},
since every finite-dimensional separable subalgebra of $K[\mathbb{G}]$ isomorphically maps into
$K[\mathbb{G}_{ev}]$, and every finite-dimensional separable subalgebra of $K[\mathbb{G}_{ev}]$
is uniquely such an isomorphic image.

We say that an affine supergroup $\mathbb{G}$ is \emph{smooth}, 
if $K[\mathbb{G}]$ is smooth in $\mathsf{SAlg}_K$, or namely, an epimorphism onto 
$K[\mathbb{G}]$ in $\mathsf{SAlg}_K$ splits
provided its kernel is nilpotent. As will be proved by
Proposition \ref{prop:smooth_Hopf} in the Appendix, this is equivalent
to saying that the associated affine group $\mathbb{G}_{ev}$ is smooth. 

In the remaining of this subsection we let $\mathbb{G}$ be an affine supergroup. 

A \emph{left $\mathbb{G}$-supermodule} may be defined to be a right $K[\mathbb{G}]$-super-comodule.
We let
\[ \mathbb{G}\text{-}\mathsf{SMod} \] 
denote the abelian symmetric category 
of left $\mathbb{G}$-supermodules. It is thus
identified with the category
$\mathsf{SComod}\text{-}K[\mathbb{G}]$ of right $K[\mathbb{G}]$-super-comodules.

Let $\mathbb{H}$ be a closed super-subgroup of $\mathbb{G}$. The \emph{right $\mathbb{H}$-adjoint action}
on $\mathbb{G}$ is defined by
\begin{equation*}\label{eq:adjoint_H-action} 
\mathbb{G}(R) \times \mathbb{H}(R)\to \mathbb{G}(R),\ (\gamma,\eta) \mapsto \eta^{-1}\gamma \eta
\end{equation*}
for each $R \in \mathsf{SAlg}_K$. This indeed gives a morphism 
$\mathbb{G} \times \mathbb{H} \to \mathbb{G}$, so that $K[\mathbb{G}]$
is made into a Hopf-algebra object in $\mathbb{H}\text{-}\mathsf{SMod}$.

Regard $K[\mathbb{G}]$ as a coalgebra, and let 
\[ C=\mathrm{Corad}(K[\mathbb{G}]) \]
denote its coradical; 
this is by definition (see \cite[Page 181]{sw}) 
the direct sum of all simple subcoalgebras of $K[\mathbb{G}]$. 

We say that $\mathbb{G}$ is \emph{linearly reductive} if $\mathbb{G}\text{-}\mathsf{SMod}$ is semisimple;
the condition is equivalent to that the coalgebra $K[\mathbb{G}]$ is cosemisimple, or $C = K[\mathbb{G}]$, by
\cite[Lemma 4]{mas3}. 
As is seen from Weissauer's classification result \cite{weissauer} 
(see also \cite[Theorem 7.1]{mas4}) and \cite[Theorem 45]{mas3}, 
those linearly reductive affine supergroups which are not affine groups are rather restricted 
in characteristic zero,
and are empty in positive characteristic. 

We say that $\mathbb{G}$ is \emph{unipotent} if the simple objects in $\mathbb{G}\text{-}\mathsf{SMod}$ are only the
obvious ones defined on $K$. This is equivalent to $C=K$.
We have the following.

\begin{prop}[\text{\cite[Theorem 41]{mas3}, \cite[Theorem 3.3]{zubul}}]\label{prop:unipotency}
$\mathbb{G}$ is unipotent if and only if $\mathbb{G}_{ev}$ is unipotent. 
\end{prop}

Due to the faithful (co)flatness results 
for Hopf superalgebras proved by the authors \cite{mas2, zub2}, we can
discuss freely, just as for affine groups, quotient supergroups of $\mathbb{G}$, including
their correspondence with normal closed super-subgroups
of $\mathbb{G}$; see \cite{zub2, maszub, mas3}. 
To be more precise, given a normal closed super-subgroup $\mathbb{N}$ of $\mathbb{G}$, 
the dur sheafification of the group-valued functor $R \mapsto \mathbb{G}(R)/\mathbb{N}(R)$ 
is representable \cite[Theorem 6.2]{zub2}. 
The thus obtained affine supergroup is denoted simply by $\mathbb{G}/\mathbb{N}$, in this paper.
This is represented by the unique Hopf super-subalgebra $B \subset K[\mathbb{G}]$ with the property
$K[\mathbb{G}]/B^+K[\mathbb{G}]=K[\mathbb{N}]$, where $B^+K[\mathbb{G}]$ is the Hopf super-ideal of 
$K[\mathbb{G}]$ generated by $B^+ =\mathrm{Ker}\, \varepsilon_B$. If $\mathbb{G}$ is an algebraic
supergroup, then $\mathbb{G}/\mathbb{N}$ is, too. 

There exists the largest unipotent closed normal super-subgroup of $\mathbb{G}$, which is
called the \emph{unipotent radical} of $\mathbb{G}$, and is denoted by $\mathbb{G}_u$. 
Since the coradical $C$ of $K[\mathbb{G}]$ 
is a super-subcoalgebra which is stable under the antipode, the ideal $C^+K[\mathbb{G}]$ generated by 
$C^+= \mathrm{Ker}(\varepsilon_{K[\mathbb{G}]}|_C)$ is a Hopf super-ideal of $K[\mathbb{G}]$. 

\begin{lemma}\label{lem:unipotent_radical}
The unipotent radical $\mathbb{G}_u$ of $\mathbb{G}$ indeed exists, and is 
represented by the quotient Hopf superalgebra
$K[\mathbb{G}]/C^+K[\mathbb{G}]$ of $K[\mathbb{G}]$.  
\end{lemma}
\begin{proof} 
Set $A=K[\mathbb{G}]$. A Hopf super-subalgebra $B \subset A$ represents 
the affine group $\mathbb{G}/\mathbb{N}$, where $\mathbb{N}$ is the closed normal
super-subgroup of $\mathbb{G}$
represented by
\begin{equation}\label{eq:quotient_Hopf}
A/B^+A. 
\end{equation}
See \cite[Theorem 5.9]{mas2}.
Slightly modifying the proof of \cite[Lemma 4]{tak}, we see that $A/B^+A$
has the property $\mathrm{Corad}(A/B^+A)=K$ (or the $\mathbb{N}$ above is unipotent)
if and only if $C\subset B$. Therefore, the Hopf superalgebra $B_C$ generated by $C$ is the smallest
Hopf super-subalgebra of $A$ such that $A/B_C^+A$ has the property. Since $B_C^+A=C^+A$, 
the desired result follows. 
\end{proof}

One sees from the proof above that $\G/\G_u$, being represented by $B_C$, has a trivial unipotent
radical, or in notation,
\begin{equation}\label{eq:trivial_uni_rad}
(\G/\G_u)_u=1. 
\end{equation}

\subsection{}\label{subsec:even-diagonalizable}
Recall that an algebraic group $G$ is said to be \emph{diagonalizable} if every $G$-module
is a direct sum of one-dimensional $G$-modules, or equivalently, if the Hopf algebra $K[G]$
is spanned by grouplikes. We say that $G$ is \emph{trigonalizable} if every simple $G$-module
is one-dimensional,
or equivalently, if the coradical $\mathrm{Corad}(K[G])$ of $K[G]$ is
spanned by grouplikes. Diagonalizable algebraic groups are trigonalizable. 

\begin{definition}\label{def:even-diagonalizable}
(1)\ Recall from Section \ref{subsec:intro2} that
an algebraic supergroup $\mathbb{G}$ is said to be \emph{even-diagonalizable} (resp., \emph{even-trigonalizable})
if the algebraic group $\mathbb{G}_{ev}$ is diagonalizable (resp., trigonalizable). 

(2)\ An even-diagonalizable supergroup $\mathbb{G}$ is said to be \emph{super-diagonalizable}, if
in addition, $\mathbb{G}_u=1$.

(3)\ An algebraic supergroup $\mathbb{G}$ is said to be \emph{super-trigonalizable}, if
$\mathbb{G}/\mathbb{G}_u$ is even- or equivalently (see \eqref{eq:trivial_uni_rad}), super-diagonalizable. 
\end{definition}

Obviously, even-diagonalizable (resp., super-diagonalizable) supergroups are even-trigonalizable
(resp., super-trigonalizable). By definition, super-diagonalizable supergroups are even-diagonalizable.

\begin{lemma}\label{lem:even_vs_super}
An algebraic supergroup $\mathbb{G}$ is super-trigonalizable if and only if it is even-trigonalizable
and $(\mathbb{G}_u)_{ev}= (\mathbb{G}_{ev})_u$. 
\end{lemma}
\begin{proof}
Given an algebraic supergroup $\mathbb{G}$, we have the short exact sequence 
$1\to \mathbb{G}_u \to \mathbb{G} \to \mathbb{G}/\mathbb{G}_u\to 1$. 
This gives rise to the short exact sequence 
$1\to (\mathbb{G}_u)_{ev} \to \mathbb{G}_{ev} \to (\mathbb{G}/\mathbb{G}_u)_{ev}\to 1$
of algebraic groups, by \cite[Theorem 5.13 (3)]{mas2}. Note that $(\mathbb{G}_u)_{ev}$ is unipotent. 
Then we see that $(\mathbb{G}/\mathbb{G}_u)_{ev}$ is diagonalizable if and only if $\mathbb{G}_{ev}$
is trigonalizable and $(\mathbb{G}_u)_{ev}= (\mathbb{G}_{ev})_u$. This proves the lemma. 
\end{proof}

Here are examples of elementary algebraic supergroups. 

\begin{example}\label{ex:algebraic_(super)group}
(1)\ Let $\mathsf{G}_m$ denote the multiplicative group, and
$\mu_n$ the closed subgroup of $n$-th roots of unity; these are diagonalizable. 
Thus, $K[\mathsf{G}_m]=K[t, t^{-1}]$, $K[\mu_n]=K[t]/(t^n-1)$, with $t$ grouplike. 

(2)\ Let $\mathsf{G}_a$ denote the additive group; this is unipotent. 
Thus $K[\mathsf{G}_a]$ is the polynomial algebra $K[t]$
with $t$ (even) primitive.

(3)\ Let $\mathsf{G}_a^-$ be the algebraic supergroup represented by the exterior algebra 
$\wedge(z)=K[z]/(z^2)$ generated by one odd primitive $z\ne 0$; this is unipotent. Note
that $\mathsf{G}_a^-(R)$ is the additive group $R_1$, where $R \in \mathsf{SAlg}_K$. 
\end{example}

\section{Harish-Chandra pairs}\label{sec:HCP}

\subsection{}\label{subsec:HCPbasics}

Let $G$ be an algebraic group. We have the Lie algebra $\mathrm{Lie}(G)$ of $G$, which is finite-dimensional, 
and a canonical pairing 
\begin{equation}\label{eq:pairing}
\langle \ , \ \rangle : \mathrm{Lie}(G) \times K[G] \to K. 
\end{equation}
We regard $\mathrm{Lie}(G)$ as a right $G$-module with respect to the structure 
which is induced from the right $G$-adjoint action on $G$; the structure also
will be called the \emph{right $G$-adjoint action}.  
To be explicit, let
$x\in \mathrm{Lie}(G)$ and $\gamma \in G(R)$, where $R$ is 
a commutative algebra. Then the result $x^{\gamma}$ of the action is determined by
\begin{equation}\label{eq:adjoint}
\langle x^{\gamma}, a \rangle 
= \gamma^{-1}(a_{(1)})\langle x, a_{(2)}\rangle\gamma(a_{(3)}), \quad
a \in K[G].
\end{equation}

Given a right $G$-module $V$, a natural right $\mathrm{Lie}(G)$-module structure
is induced on $V$. We present this as $v \triangleleft x$, where $v \in V$, $x\in \mathrm{Lie}(G)$. 
The element is explicitly given by
\[ v \triangleleft x = \langle x, v_{(-1)}\rangle \, v_{(0)}, \]
where $v \mapsto v_{(-1)}\otimes v_{(0)},\ V \to K[G]\otimes V$ denotes the left $K[G]$-comodule
structure on $V$ which gives the right $G$-module structure.

The concept of \emph{Harish-Chandra pairs} goes back to Kostant \cite{kostant}. The definition below
is reproduced from \cite{mas4}. Its presentation is slightly different from the one in fashion found 
in \cite{ccf, carfi, gavarini, vish}, and is indeed a translation of \cite[Definition 7]{mas3} given in Hopf-algebraic terms. 

\begin{definition}\label{def:HCP}
A \emph{Harish-Chandra pair} is a pair $(G,V)$ of an algebraic group $G$ and a finite-dimensional
right $G$-module $V$, which is given 
a $G$-equivariant bilinear map $[\ , \ ] : V \times V \to \mathrm{Lie}(G)$ 
such that
\begin{itemize}
\item[(i)] $[v,v']=[v',v]$, 
\item[(ii)] $v \triangleleft[v,v]= 0$, 
\end{itemize}
where $v,v' \in V$. 
\end{definition}

The bilinear map above is called the \emph{bracket} associated with the pair. All Harish-Chandra pairs
form a category $\mathsf{HCP}$. A morphism $(\phi, \psi) : (G_1,V_1) \to (G_2,V_2)$ consists
of a morphism $\phi : G_1 \to G_2$ of algebraic groups and a linear map $\psi : V_1\to V_2$ such that
\begin{itemize}
\item[(a)] $\psi$ is $G_1$-equivariant, with $V_2$ regarded as a $G_1$-module through $\phi$,
\item[(b)] $[\psi(v), \psi(v') ] =\mathrm{Lie}(\phi)([v, v'])$, \quad $v, v' \in V $.
\end{itemize}

Let $\mathsf{ASG}$ denote the category of algebraic supergroups. It is anti-isomorphic to the category
of finitely generated super-commutative Hopf superalgebras.

Let $\mathbb{G} \in \mathsf{ASG}$. 
We have the Lie superalgebra $\mathrm{Lie}(\mathbb{G})$ of $\mathbb{G}$, which is finite-dimensional, 
and a canonical pairing, $\langle \ , \ \rangle : \mathrm{Lie}(\mathbb{G}) \times K[\mathbb{G}] \to K$, analogous
to \eqref{eq:pairing}; see \cite[Section 4.1]{masshiba}, for example. 
Set
\[ G = \mathbb{G}_{ev}, \quad V = \mathrm{Lie}(\mathbb{G})_1 . \]
The latter $V$ is the odd component of $\mathrm{Lie}(\mathbb{G})$. Regard this $V$ as a right 
$G$-module with respect to the $G$-action defined by the formula
\begin{equation}\label{eq:analogous_adjoint}
\langle v^{\gamma}, a \rangle 
= \gamma^{-1}(\overline{a}_{(1)}) \langle v, a_{(2)}\rangle \gamma(\overline{a}_{(3)}), \quad
v \in V,\ \gamma \in G(R),\ a \in K[\mathbb{G}]
\end{equation}
analogous to \eqref{eq:adjoint}, where $a \mapsto \overline{a}$, $K[\mathbb{G}]\to K[G]$ denotes
the quotient map. Since the even component $\mathrm{Lie}(\mathbb{G})_0$ of $\mathrm{Lie}(\mathbb{G})$
coincides with $\mathrm{Lie}(G)$, 
the structure map of $\mathrm{Lie}(\mathbb{G})$ restricts to
\[ [\ , \ ] : V \times V \to \mathrm{Lie}(G). \]

\begin{theorem}\label{thm:equivalence}
Given the last map as the bracket, $(G, V)$ is a Harish-Chandra pair.
The assignment $\mathbb{G} \mapsto
(G, V)$ above is functorial, and gives a category equivalence $\mathsf{ASG}\to \mathsf{HCP}$.  
\end{theorem}

This is a reformulation of \cite[Theorem 29]{mas3}, which formulated the result
in Hopf-algebraic terms; see also \cite{masshiba}, especially Remarks 4.5 and 4.27 therein.
Carmeli and Fioresi \cite[Theorem 3.12]{carfi} proved the result essentially as formulated above,  
when $K=\mathbb{C}$. 

\begin{rem}\label{rem:quasi-inverse}
A quasi-inverse of the equivalence $\mathbb{G}\mapsto (G,V)$ is explicitly constructed in 
\cite[Section 4.6]{mas3}. 
Some details of the construction will be given in the proof of Proposition \ref{prop:Ggx} in
a special situation. 
We remark here that the superalgebra structure of $K[\mathbb{G}]$
recovers from the Harish-Chandra pair $(G,V)$ assigned to $\mathbb{G}$, so simply as
\begin{equation}\label{eq:tensor_decomposition}
K[\mathbb{G}] = K[G] \otimes \wedge(V)\ \text{as a superalgebra}. 
\end{equation} 
It follows that an algebraic supergroup $\mathbb{G}$ is purely even, or it is an algebraic group, 
if and only if $\mathrm{Lie}(\mathbb{G})$ is purely even. 
The isomorphism \eqref{eq:tensor_decomposition} can be chosen, in fact, so as 
to preserve the left $K[G]$-comodule structure and the counit, as well; see \cite[Theorem 4.5]{mas2}. 
\end{rem}

\begin{rem}\label{rem:equivalence_generalized}
Theorem 4.23 of \cite{masshiba} generalizes the equivalence $\mathsf{ASG}\to \mathsf{HCP}$ above, 
replacing $K$ with a commutative ring, say $R$, which is $2$-\emph{torsion free} in the sense that
$2 : R \to R$ is injective; see also \cite{gavarini, masshiba2}. 
Working over such a ring one poses in \cite{masshiba}
some additional assumptions to the objects, to define
$\mathsf{ASG}$ and $\mathsf{HCP}$; it would, however, deserve to remark that the splitting property
\eqref{eq:tensor_decomposition}, which is included in the added assumptions, was proved to hold,
very recently by \cite{masshiba2}, and so it is needless to assume. 
The cited theorem of \cite{masshiba}
will be used in the proof of
Proposition \ref{prop:center_ev}.  
\end{rem}

\subsection{}\label{subsec:normalsubHCP}
Proposition 34 of \cite{mas3} characterizes those sequences in $\mathsf{HCP}$ which correspond to
short exact sequences in $\mathsf{ASG}$. We will reproduce below the result as two lemmas,
in a suitable form for the subsequent argument.

Let $\mathbb{G}$ be an algebraic supergroup, and 
$(G, V)$ the corresponding Harish-Chandra pair. Let $(H,W)$ be a sub-object of the pair; 
it consists of a closed subgroup $H \subset G$ and an
$H$-submodule $W \subset V$ such that the bracket associated with $(G,V)$ satisfies
$[W,W] \subset \mathrm{Lie}(H)$. It is assigned uniquely to a closed super-subgroup, say $\mathbb{H}$,
of $\mathbb{G}$. 

\begin{lemma}\label{lem:normal}
$\mathbb{H}$ is normal in $\mathbb{G}$ if and only if
\begin{itemize}
\item[(i)] $H$ is normal in $G$, 
\item[(ii)]  $W$ is $G$-stable in $V$, 
\item[(iii)] the induced $G$-module action on $V/W$, restricted to $H$, is trivial, and
\item[(iv)] $[V, W] \subset \mathrm{Lie}(H)$. 
\end{itemize}
\end{lemma}

Suppose the conditions above are satisfied. Then by (i)--(iii), we have the quotient algebraic group $G/H$, and
$V/W$ naturally turns into a right $G/H$-module. 
By (iv), the bracket above induces 
$[\ , \ ] : V/W\times V/W \to \mathrm{Lie}(G)/\mathrm{Lie}(H) \subset \mathrm{Lie}(G/H)$. 

\begin{lemma}\label{lem:quotientHCP}
$(G/H, V/W)$, given the induced bracket above, forms a Harish-Chandra pair, which is assigned to the
quotient algebraic supergroup $\mathbb{G}/\mathbb{H}$. 
\end{lemma}

\subsection{}\label{subsec:characterize_super-diagonalizable}
To show a first application of Harish-Chandra pairs, we let 
$\mathbb{G}$ be an algebraic supergroup with the corresponding Harish-Chandra pair $(G,V)$. 
Let $\lambda : V \to K[G]\otimes V$ denote the left $K[G]$-comodule structure which gives the right
$G$-module structure on $V$. 

Given a normal closed subgroup $H$ of $G$, let $V_H$ denote the smallest
$G$-submodule of $V$ such that $V/V_H$ is trivial as an $H$-module. In the dual $G$-module $V^*$,
the dual vector space $(V/V_H)^*$ is the pullback
of $K[G/H] \otimes V^*$ along the dual $K[G]$-comodule structure $\lambda^* : V^* \to K[G]\otimes V^*$.
If $H$ is connected, then we have $V_H=\mathrm{Dist}(H)^+\, V$, where $\mathrm{Dist}(H)^+$ denotes
the augmentation ideal of the distribution algebra $\mathrm{Dist}(H)$ of $H$. 

\begin{lemma}\label{lem:trivial_unipotent_radical}
$\mathbb{G}_u$ is trivial, $\mathbb{G}_u=1$, if and only if
\begin{itemize}
\item[(i)] for every non-trivial, unipotent normal closed subgroup $H$ of $G$, 
$[V,V_H]$ is not included in $\mathrm{Lie}(H)$, and 
\item[(ii)] there exists no non-zero $G$-submodule $W$ of $V$ such that $[V,W]=0$. 
\end{itemize}
\end{lemma}
\begin{proof}
As is seen from Proposition \ref{prop:unipotency} and Lemma \ref{lem:normal},
Condition (i) (resp., (ii)) is equivalent to 
saying that $(G,V)$ does not include a sub-object $(H, W)$ with $H \ne 1$ (resp., $H=1$) 
which gives rise to a non-trivial, unipotent normal closed super-subgroup of $\mathbb{G}$. 
\end{proof}

Condition (i) is superfluous if the unipotent radical $G_u$ of $G$ is trivial. This is the case
if $G$ is diagonalizable, as will be assumed below.

Assume that $\mathbb{G}$ is even-diagonalizable, or $G$ is diagonalizable. Then $K[G]$ 
is spanned by the character group $X=X(G)$ of $G$, and we have
\[ V = \bigoplus_{g\in X} V(g),\]
where $V(g) = \{ v \in V\mid \lambda(v)=g \otimes v\}$. Since $G$ acts trivially on $\mathrm{Lie}(G)$,
it follows that $[V(g), V(h)]=0$ unless $gh=1$. 
Lemma \ref{lem:trivial_unipotent_radical} now implies the following.  

\begin{prop}\label{prop:characterize_super-diagonalizable}
$\mathbb{G}$ is super-diagonalizable, if and only if the bracket $[\ , \ ] : V \times V \to \mathrm{Lie}(G)$
is non-degenerate, if and only if for every $g\in X$ such that $V(g) \ne 0$, 
the restriction $[\ , \ ]|_{V(g) \times V(g^{-1})}$ is non-degenerate. 
\end{prop}

\subsection{}\label{subsec:abelian_supergroup}
Let $\mathsf{AbelASG}$ denote the full subcategory of $\mathsf{ASG}$ which consists of all abelian
algebraic supergroups. Let $\mathsf{AbelAG}$ denote the category of abelian algebraic groups, and
$\mathsf{Vec}_K$ the category of finite-dimensional vector spaces.

Given $V \in \mathsf{Vec}_K$, regard the exterior algebra $\wedge(V^*)$ on the dual vector space $V^*$
of $V$ as a Hopf superalgebra in which every element in $V^*$ is an odd primitive, and let $(\mathsf{G}_a^-)^V$
denote the corresponding abelian algebraic supergroup. This last is isomorphic to the product
$(\mathsf{G}_a^-)^{\dim V}$ of $\dim V$ copies of $\mathsf{G}_a^-$. 

\begin{prop}\label{prop:abelian_supergroup}
There is a category equivalence 
\[ \mathsf{AbelASG} \approx \mathsf{AbelAG}\times \mathsf{Vec}_K \]
between $\mathsf{AbelASG}$ and the product $\mathsf{AbelAG}\times \mathsf{Vec}_K$,
which assigns $G \times (\mathsf{G}_a^-)^V$ to each object $(G,V)$ in the product category.  
\end{prop}

This is an easy consequence of Theorem \ref{thm:equivalence}, or a direct consequence of
\cite[Theorem 3.16]{mas2}. 

\section{The algebraic supergroups $G_{g,x}$}\label{sec:Ggx}

Let $G$ be an algebraic group. Let $g \in K[G]$ and $x \in \mathrm{Lie} (G)$ such that
\begin{itemize}
\item[(i)] $g$ is grouplike, and
\item[(ii)]
the right $G$-adjoint action on $x$ arises from $x \mapsto g^2 \otimes x$. 
\end{itemize}
The last condition is equivalent to saying that
\[ x^{\gamma} = x \otimes \gamma(g)^2\ \, \text{in}\ \, V \otimes R,\quad \gamma \in G(R), \]
where $R$ is an arbitrary commutative algebra. 

\begin{rem}\label{rem:abelian_case}
Suppose $x \ne 0$. If the $G$-action on $x$ is trivial, then Condition (ii) implies $g^2=1$. 
This is the case if $G$ is abelian.
\end{rem}

Suppose that $G$, $g$ and $x$ are as before Remark \ref{rem:abelian_case}. 

\begin{lemmadef}\label{lemdef:Ggx}
Let $V=Kv$ be the one-dimensional right $G$-module defined by the
left $K[G]$-comodule structure $v \mapsto g \otimes v$. Then $(G, V)$, given the bracket
determined by $[v,v]=2x$, is a Harish-Chandra pair.

\emph{We let} $G_{g,x}$ \emph{denote the corresponding algebraic supergroup.}
\end{lemmadef}
\begin{proof}
Condition (ii) above ensures that $[\ , \ ] : V \times V \to \mathrm{Lie}(G),\ [v,v]=2x$ is $G$-equivariant. 
It remains to prove that $v \triangleleft [v,v]=0$, or equivalently
\begin{equation}\label{eq:xg0}
\langle x, g \rangle = 0. 
\end{equation}
Since the right $G$-adjoint action on $\mathrm{Lie}(G)$ induces the right 
$\mathrm{Lie}(G)$-adjoint action, we have
\[ [x, y] = \langle y, g^2\rangle \, x = 2\langle y, g\rangle \, x,\quad y \in \mathrm{Lie}(G). \]
This, applied to $y=x$, proves \eqref{eq:xg0} since $[x,x]=0$. 
\end{proof}

\begin{prop}\label{prop:Ggx}
Let $G_{g,x}$ be an algebraic supergroup constructed above.
\begin{itemize} 
\item[(1)] Set $B=K[G]$. 
Then the Hopf superalgebra $K[G_{g,x}]$ is the tensor product $B \otimes \wedge(z)$
of $B$ with the
exterior algebra $\wedge(z)=K[z]/(z^2)$ generated by one odd element $z\ne 0$. The structure maps
$\Delta$, $\varepsilon$ and $\mathcal{S}$ are determined by
\begin{gather*}
\Delta(b)=\Delta_B(b) +\langle x, b_{(2)} \rangle \, \Delta_B(b_{(1)})(z\otimes gz),\\
\Delta(z)=1\otimes z +z\otimes g,\\
\varepsilon(b)=\varepsilon_B(b),\quad \varepsilon(z)=0,\\
\mathcal{S}(b)=\mathcal{S}_B(b),\quad \mathcal{S}(z)=-g^{-1}z,
\end{gather*}
where $b \in B$, $\Delta_B(b)=b_{(1)}\otimes b_{(2)}$, and $\langle x,b \rangle$
denotes the pairing \eqref{eq:pairing}. 
In particular, $g$ remains grouplike in $K[G_{g,x}]$, and $z$ is 
an odd skew-primitive in $K[G_{g,x}]$. 
\item[(2)] $G_{g,x}$ is an algebraic supergroup $\mathbb{G}$ such that 
\[ \mathbb{G}_{ev}=G,\quad \dim\mathrm{Lie}(\mathbb{G})_1=1. \]
Conversely, every algebraic supergroup with this property is isomorphic to $G_{g,x}$ for some
$g$ and $x$. 
\end{itemize}
\end{prop}
\begin{proof}
(1)\ 
Let $J = B^{\circ}$ be the dual Hopf algebra \cite[Page 122]{sw} of $B$. Let
\begin{equation}\label{eq:extended_pairing}
\langle \ , \ \rangle : J \times B \to K 
\end{equation}
denote the canonical pairing. 
This extends the pairing \eqref{eq:pairing} since $\mathrm{Lie}(G) = P(J)$. 
Note that $V=Kv$ is a right $J$-module, naturally defined by
$v \triangleleft s = \langle s, g \rangle \, v$, where $s \in J$.  
The pair $(J, V)$, given the bracket above, forms a dual Harish-Chandra pair \cite[Definition 6]{mas3}. 
Let $H=J\otimes \wedge(v)$ denote the super-coalgebra defined as the tensor product
of the coalgebra $J$ with $\wedge(v)=K1\oplus Kv$; in this last super-coalgebra,
$1$ is supposed to be grouplike, and $v$ an odd primitive. By \cite[Theorem 10 (1), Lemma 11]{mas3}
the dual Harish-Chandra pair constructs on $H$
a Hopf superalgebra, with respect to the product defined by
\[ 
(s \otimes v^i)(t \otimes v^j) = 
\begin{cases}
st \otimes v^j & i=0, \\
\langle t_{(2)},g \rangle \, st_{(1)} \otimes v  & i=1, j=0,\\
\langle t_{(2)},g \rangle \, st_{(1)}x \otimes 1 & i=j=1,
\end{cases}
\]
where $s, t \in J$, and $\Delta_J(t)=t_{(1)}\otimes t_{(2)}$; the unit is $1\otimes 1$. 
We remark that in $H$, $[v,v]$ coincides with $2x$, or more precisely,
$[1\otimes v, 1 \otimes v]=2(1\otimes v)^2$ coincides with $2(x\otimes 1)$.

Set $A = K[G_{g,x}]$. By Remark \ref{rem:quasi-inverse},  
$A = B \otimes \wedge(z) \, (= B \oplus Bz)$ as a superalgebra.
Extend the pairing \eqref{eq:extended_pairing} to $\langle \ , \ \rangle : H \times A \to K$ so that
\[ 
\langle s \otimes v^i, b \otimes z^j\rangle = \delta_{ij} \langle s,b\rangle,
\quad s \in J,\ b \in B,\ i, j \in \{ 0,1 \},
\]
By \cite[Proposition 28 (2)]{mas3} the extended one is
a Hopf pairing, or in other words, there is induced a Hopf-superalgebra map $A \to H^{\circ}$; 
see \cite[Remark 1]{mas3}.
The last map is injective since the natural map $B \to J^*$ is injective by \cite[Theorem 6.1.3]{sw}. 
Therefore, we can dualize
the structures on $H$, to obtain the structures on $A$. For example, given $b \in B$, we have
$\Delta(b) \in (B \otimes B) \oplus (B\otimes B)(z\otimes z)$. 
The component in $(B\otimes B)(z\otimes z)$
is seen to be as stated above, from the computation
\begin{align*}
\langle (s\otimes v)(t\otimes v),\Delta(b)\rangle 
&=\langle t_{(2)},g \rangle \, \langle st_{(1)}, b_{(1)}\rangle\, \langle x, b_{(2)} \rangle \\
&=\langle s, b_{(1)}\rangle\, \langle t, b_{(2)}g \rangle\, \langle x, b_{(3)} \rangle\\
&=\langle s \otimes v, b_{(1)}\otimes z\rangle\, \langle t \otimes v, b_{(2)}g\otimes z\rangle
\, \langle x, b_{(3)} \rangle. 
\end{align*}
For $\Delta(z)$, note that $\Delta(z) \in (B\otimes Bz) \oplus (Bz \otimes B)$. 
The component in $Bz \otimes B$ is seen to be as stated above, from the computation
\begin{align*}
\langle (s\otimes v)(t\otimes 1),\Delta(z)\rangle 
&=\langle t_{(2)},g \rangle \, \langle st_{(1)}, 1 \rangle 
=\langle t, g \rangle\, \langle s, 1 \rangle \\
&=\langle s \otimes v, 1 \otimes z\rangle\, \langle t \otimes 1, g\otimes 1 \rangle. 
\end{align*}
The counit is easy to see.
For the antipode note that the equations above define a superalgebra endomorphism on $A$.
Then one sees that it indeed satisfies the axiom of antipodes. 

We see from \eqref{eq:xg0} that $g$ is grouplike in $A$. Obviously, $z$ is skew-primitive. 

(2)\ The algebraic supergroups with the prescribed property correspond precisely to the Harish-Chandra pairs
$(G, V)$ with $\dim V =1$. This implies the first half. Suppose that $Kv$ is a one-dimensional right $G$-module, 
whose structure is given uniquely by a grouplike, say $g$, in $K[G]$. 
If $(G, Kv)$ is a Harish-Chandra pair, then the $g$ and $x =\frac{1}{2}[v,v]$ must satisfy
(ii), so that the pair must be as above. This proves the second half. 
\end{proof}

\begin{prop}\label{prop:classify_Ggx}
Let $(G_i)_{g_i,x_i}$, $i=1,2$, be algebraic supergroups constructed as above.  
An isomorphism $(G_1)_{g_1,x_1} \overset{\simeq}{\longrightarrow} (G_2)_{g_2,x_2}$ arises uniquely from
an isomorphism $\phi : G_1 \overset{\simeq}{\longrightarrow} G_2$ and an element $\alpha \in K \setminus 0$,
such that $K[\phi] : K[G_2] \to
K[G_1]$ sends $g_2$ to $g_1$, and $\mathrm{Lie}(\phi)(x_1) = \alpha^2 x_2$.  
\end{prop}
\begin{proof}
Every isomorphism between the algebraic supergroups arises uniquely from an isomorphism 
$(G_1, Kv_1) \overset{\simeq}{\longrightarrow} (G_2, Kv_2)$ between the corresponding Harish-Chandra pairs.
The latter is a pair $(\phi, \alpha)$ of an isomorphism $\phi : G_1 \overset{\simeq}{\longrightarrow} G_2$
and a scalar $\alpha\ne 0$ giving $v_1 \mapsto \alpha v_2$. We see that
the pair indeed gives a map of Harish-Chandra pairs if and only if $K[\phi](g_2)=g_1$ and 
$\mathrm{Lie}(\phi)(x_1) = \alpha^2 x_2$. This proves the proposition. 
\end{proof}

\begin{example}\label{ex:Ggx}
Let $G$ be an algebraic group. 

(1)\ Choose $0$ as the $x$ in $\mathrm{Lie}(G)$. Then any grouplike $g \in K[G]$ satisfies (ii). We see
that $G_{g,0}$ is the semi-direct product $G\ltimes \mathsf{G}_a^-$ with respect to the action arising from
$z \mapsto z \otimes g$. 

(2)\ Suppose $G = \mu_n$, $n \ge 1$. Since $\mathrm{Lie}(G)=0$, the possible $G_{g,x}$ are 
the $\mu_n \ltimes \mathsf{G}_a^-$ as above.

(3)\ Suppose $G = \mathsf{G}_m$ or $\mathsf{G}_a$, so that $K[G]=K[t,t^{-1}]$ or $K[t]$. 
Then $\mathrm{Lie}(G)$ is spanned by the specific element $y$ determined by
$\langle y, t\rangle = 1$. Choose a non-zero $\lambda y$, $\lambda \in K\setminus 0$, as the $x$.
Then the $g$ must be 1 by the requirement $g^2=1$ from (ii). 
Given two $G_{1,\lambda_iy}$, $i=1,2$, Proposition \ref{prop:classify_Ggx} shows the following. 

Case $G = \mathsf{G}_m$; $G_{1,\lambda_1y}\simeq G_{1,\lambda_2y}$ if and only if
$\lambda_1/\lambda_2$ or $-\lambda_1/\lambda_2$ is the square $\alpha^2$
of some $\alpha \in K\setminus 0$.  

Case $G = \mathsf{G}_a$; the two are necessarily isomorphic. If $K$ is the field
$\mathbb{R}$ of real numbers,
$(\mathsf{G}_a)_{1,y}$ coincides with the $\mathbb{R}^{1|1}$ given in \cite[Page 74]{dm}
and \cite[Page 277]{varadarajan}.

(4)\ Concerning $G_{g,x}$ we see from (2), (3) that if $g^{\pm 1}$ generate $K[G]$,  
then $G_{g,x}$ isomorphic to one of the following
\begin{equation}\label{eq:gen_by_1skewprim}
\mathsf{G}_a^-,\quad \mathsf{G}_m \ltimes \mathsf{G}_a^-,\quad \mu_n \ltimes \mathsf{G}_a^-.
\end{equation}
\end{example}

\section{Representations of $D_{g, x}$}\label{sec:Dgx}
 
In the situation of Section \ref{sec:Ggx}
we suppose that $G$ is a diagonalizable 
algebraic group $D$, and study representations of $D_{g,x}$. 
The algebraic supergroups $D_{g,x}$ are characterized as the
even-diagonalizable supergroups $\mathbb{G}$ such that $\dim\mathrm{Lie}(\mathbb{G})_1=1$. 

Let $D$ be a diagonalizable algebraic supergroup, and choose arbitrarily 
elements $g \in K[D]$ and $x \in \mathrm{Lie}(D)$ which
satisfy Conditions (i), (ii) at the beginning of Section \ref{sec:Ggx}. 
Given $M \in D_{g,x}\text{-}\mathsf{SMod}$, let $\Pi M$ denote the parity shift of $M$,
so that $(\Pi M)_i=M_{i+1}$, $i \in \mathbb{Z}_2$. 


Let $X = X(D)$ denote the character group of $D$. Then
$K[D]$ is the group algebra $KX$ on $X$, and $g \in X$.  
Recall $K[D_{g,x}] = KX \oplus KXz$,
and 
\[ \Delta(h) = h \otimes h + \langle x, h \rangle \, hz \otimes gh z, \quad 
\Delta(hz) = h \otimes hz + hz \otimes gh, \]
where $h \in X$. Let
\[ Y = \{ h \in X \mid \langle x, h \rangle = 0 \}. \]
By \eqref{eq:xg0},  $Y$ is a subgroup of $X$ containing $g$. Define in $K[D_{g,x}]$,
\[ L(h) = Kh \oplus Khz,\ h \in X;\quad S(h) = Kh,\ h\in Y. \]
These are all right $K[D_{g,x}]$-super-subcomodules, or $D_{g,x}$-super-submodules, of $K[D_{g,x}]$, 
and we have
\begin{equation}\label{eq:direct_sum}
K[D_{g,x}]= \bigoplus_{h \in X}L(h).
\end{equation} 

\begin{prop}\label{prop:Dgx_representation}
In $D_{g,x}\text{-}\mathsf{SMod}$ we have the following.
\begin{itemize}
\item[(1)] All indecomposable objects are given by
\begin{equation}\label{eq:indecomposable}
L(h),\ \, \Pi L(h), \ \, h \in X;\quad
S(h),\ \, \Pi S(h),\ \, h \in Y, 
\end{equation}
which are mutually non-isomorphic. 
\item[(2)] Among the object above the injective indecomposables are
\begin{equation}\label{eq:injective_indecomposable}
L(h),\ \, \Pi L(h), \ \, h \in X,  
\end{equation}
while the simples are
\begin{equation}\label{eq:simple}
L(h),\ \, \Pi L(h), \ \, h \in X \setminus Y;\quad
S(h),\ \, \Pi S(h),\ \, h \in Y.
\end{equation}
\end{itemize}
\end{prop}
\begin{proof}
It is easy to see that 
those listed in \eqref{eq:indecomposable} are mutually non-isomorphic. For example, 
if $f : L(h) \to L(h')$ is an isomorphism which is required to preserve the parity, 
then $f(h) = \lambda_1h'$ and $f(hz)=\lambda_2h'z$ for
some $\lambda_1, \lambda_2 \in K \setminus 0$, from which one sees $h=h'$. 

We see from \eqref{eq:direct_sum} that all injective indecomposables are given by 
\eqref{eq:injective_indecomposable}. Their socles give all simples. 
Suppose $h \in Y$. Then we have an extension
\begin{equation}\label{eq:extension_L(h)}
0 \to S(h) \to L(h) \to \Pi S(gh) \to 0,
\end{equation}
which is non-split since the odd $hz$ does not span a $D_{g,x}$-super-submodule. 
Therefore, the socle $\mathrm{Soc}\, L(h)$ equals $S(h)$.
Suppose $h \in X \setminus Y$. Since neither of the homogeneous $h$ and $hz$ in $L(h)$
spans a $D_{g,x}$-super-submodule, we have $\mathrm{Soc}\, L(h)=L(h)$. 
It follows that all simples are given by \eqref{eq:simple}. 

At the end of this section we will prove that every non-simple indecomposable is injective, 
which will complete the proof. 
\end{proof}

\begin{rem}\label{rem:inhomogeneous_grouplike}
Suppose $g=1$ and $h \in X \setminus Y$. Let $\alpha = \langle x, h \rangle \, (\ne 0)$, and assume 
$\sqrt{\alpha} \in K$. Then the elements $h \pm \sqrt{\alpha}\, hz$ contained in the simple $L(h)$ 
are grouplikes in the usual, wider sense. 
But, being inhomogeneous, each of them does not span a $D_{g,x}$-super-submodule.
\end{rem}

Before continuing the proof by Proposition \ref{prop:Dgx_Ext} below, we give consequences of 
Proposition \ref{prop:Dgx_representation}. 

\begin{corollary}\label{cor:Dgx_representation}
For $D_{g,x}$ we have the following.
\begin{itemize}
\item[(1)] $D_{g,x}$ is not linearly reductive.
\item[(2)] $D_{g,x}$ is super-diagonalizable (or $(D_{g,x})_u =1$) if and only if $x \ne 0$.
\item[(3)] $D_{g,x}/(D_{g,x})_u$ is purely even and diagonalizable if and only if $x = 0$.
\end{itemize}
\end{corollary}
\begin{proof} Indeed, Parts 1 and 3 are direct consequences of the proposition. Note that
the first condition of 3 is equivalent to that $\mathrm{Corad}(K[D_{g,x}])$ is spanned
by grouplikes, which is seen to be equivalent to $X=Y$. Part 2 follows from Proposition 
\ref{prop:characterize_super-diagonalizable}. 
\end{proof}

\begin{prop}\label{prop:Dgx_Ext}
Let $S, T \in D_{g,x}\text{-}\mathsf{SMod}$ be simples. The 1st extension space $\mathrm{Ext}^1_{D_{g,x}}(S,T)$ is 
as follows. 
\begin{itemize}
\item[(1)] If (i)~$(S, T)=(\Pi S(gh), S(h))$ or (ii)~$(S,T)= (S(gh), \Pi S(h))$, where $h \in Y$, then
\[ \mathrm{Ext}^1_{D_{g,x}}(S,T) = K. \]
Every non-split extension is isomorphic, up to a scalar-multiplication, to 
\eqref{eq:extension_L(h)} in Case (i), and to the parity shift
$0 \to \Pi S(h) \to \Pi L(h) \to S(gh) \to 0$ of \eqref{eq:extension_L(h)} in Case (ii).
\item[(2)] In the remaining cases we have
\[ \mathrm{Ext}^1_{D_{g,x}}(S,T) = 0. \]
\end{itemize}
\end{prop}
\begin{proof}
Let $0 \to T \to M \to S \to 0$ be a non-split extension. Then $\mathrm{Soc}\, M = T$. 
Let $L$ be an injective hull of $M$. Then $L$ is indecomposable. It follows from the last proof that
$\dim L =2$, and so $M=L$. Moreover, $(S, T)$ must be in Case (i) or (ii), and we have the commutative diagram
\[
\begin{xy}
(-1,0)   *++{0}  ="1",
(9,0)   *++{T}  ="2",
(20,0)   *++{M}  ="3",
(31,0)   *++{S}  ="4",
(41,0)   *++{0}  ="5",
(-1,-12)  *++{0}  ="6",
(9,-12) *++{T}  ="7",
(20,-12) *++{L}  ="8",
(31,-12) *++{S}  ="9",
(41,-12) *++{0,}  ="10",
{"1" \SelectTips{cm}{} \ar @{->} "2"},
{"2" \SelectTips{cm}{} \ar @{->} "3"},
{"3" \SelectTips{cm}{} \ar @{->} "4"},
{"4" \SelectTips{cm}{} \ar @{->} "5"},
{"6" \SelectTips{cm}{} \ar @{->} "7"},
{"7" \SelectTips{cm}{} \ar @{->} "8"},
{"8" \SelectTips{cm}{} \ar @{->} "9"},
{"9" \SelectTips{cm}{} \ar @{->} "10"},
{"2" \SelectTips{cm}{} \ar @{=} "7"},
{"3" \SelectTips{cm}{} \ar @{=} "8"},
{"4" \SelectTips{cm}{} \ar @{->}^{\simeq} "9"}
\end{xy}
\]
where the second row is \eqref{eq:extension_L(h)} or its parity shift. 
Since the isomorphism $S \overset{\simeq}{\longrightarrow} S$ above is a scalar multiplication, the 
proposition follows. 
\end{proof}

\begin{lemma}\label{lem:projective}
The injective indecomposables listed in \eqref{eq:injective_indecomposable} are projective.
\end{lemma}
\begin{proof}
Let $h \in X$. The map 
$\langle \ , \ \rangle : \Pi L(h^{-1})\times L(g^{-1}h)  \to K$ defined by
\[ \langle h^{-1}, g^{-1}hz \rangle = \langle h^{-1}z , g^{-1}h \rangle =1,\quad
\langle h^{-1} , g^{-1}h \rangle = \langle h^{-1}z, g^{-1}hz \rangle = 0 \]
is a non-degenerate bilinear form. This is, regarded as  $\Pi L(h^{-1}) \otimes L(g^{-1}h) \to K$, a morphism in 
$D_{g,x}\text{-}\mathsf{SMod}$, as is seen by using \eqref{eq:xg0}. Therefore,
$\Pi L(h^{-1})$ and $L(g^{-1}h)$ are dual to each other. This implies the desired result. 
\end{proof}

\begin{rem}\label{rem:coFrobenius}
An unpublished result  
by the first-named author, Pastro and Shibata announced in \cite{mas4} (see Lemma 7.3 and Proposition 7.5)
tells us that given an algebraic supergroup $\mathbb{G}$, every injective object in $\mathbb{G}\text{-}\mathsf{SMod}$ 
is projective if and only if $\mathbb{G}_{ev}$ has the same property, that is, 
every injective $\mathbb{G}_{ev}$-module
is projective. This is obviously the case if $\mathbb{G}$ is even-diagonalizable. Lemma \ref{lem:projective}
is a special case of this more general result. 
\end{rem}

\begin{proof}[Proof of Proposition \ref{prop:Dgx_representation}~(Continued)]
It remains to prove that every non-simple indecomposable is injective. This will follow
if we prove:
\medskip

\noindent
\textbf{Claim.}\quad \emph{Every non-zero non-semisimple object $L \in D_{g,x}\text{-}\mathsf{SMod}$ (possibly
of infinite dimension) includes a non-zero injective direct summand.}  
\medskip

We wish to prove that if $0\to T \to M \to S \to 0$ is a non-split extension, 
where $S$ is simple and $T$ is finite-dimensional semisimple, then 
$M$ includes a non-zero injective direct summand. 
Apply this to a finite-dimensional submodule $M \subset L$ and its socle $T =\mathrm{Soc}\, M$,
such that $M/T$ is simple. Then the claim will be proven. 
Suppose that $T = \bigoplus_iT_i$ with $T_i$ finitely many simples. Then for some $i$, the extension
$0 \to T_i \to M_i \to S\to 0$ induced by the projection $T \to T_i$ 
is non-split, whence $M_i$ is one of those listed in \eqref{eq:injective_indecomposable}, by 
Proposition \ref{prop:Dgx_Ext}. We have a surjection $M \to M_i$, which necessarily splits
by Lemma \ref{lem:projective}. This
proves the desired result.
\end{proof}

\section{Solvability of even-trigonalizable supergroups}\label{sec:main_results}

To prove our first main result, Theorem \ref{thm:even-trigonalizable}, we start with the following.

\begin{lemma}\label{lem:quotient}
Let $\mathbb{G}$ be a even-trigonalizable supergroup. If $\mathbb{G}$ is not purely even, then it has a quotient
supergroup which is isomorphic to one of those listed in \eqref{eq:gen_by_1skewprim}. 
\end{lemma}
\begin{proof}
Let $G =\mathbb{G}_{ev}$, $V = \mathrm{Lie}({\mathbb{G}})_1$, and 
$(G,V)$ the Harish-Chandra pair assigned to $\mathbb{G}$. 
By Remark \ref{rem:quasi-inverse} the assumption implies $V \ne 0$. 
We have the left $G$-module $V^*(\ne 0)$ which is dual to $V$. Since 
$G$ is trigonalizable, there exist $z \in V^* \setminus 0$ and a grouplike $g \in K[G]$
such that the left $G$-action on $z$ arises from $z \mapsto z \otimes g$. Define $W =(V^*/Kz)^*$.
This is a $G$-submodule of $V$ of codimension 1. We see that $K[G]/(g-1)$ is a quotient Hopf
algebra of $K[G]$. Moreover, the corresponding closed subgroup $H$ of $G$ is normal, and
$G/H$ is represented by the Hopf subalgebra $K[g^{\pm 1}]\subset K[G]$ generated by $g^{\pm 1}$. 
Regard $W$ as a right $H$-module by restriction. 
\medskip

\noindent
\textbf{Claim.}\quad \emph{$(H, W)$ is a sub-object
of the Harish-Chandra pair $(G,V)$, and the corresponding closed super-subgroup $\mathbb{H}$ of $\mathbb{G}$ 
is normal.}
\medskip
 
This will follow if we verify Conditions (iii) and (iv) of Lemma \ref{lem:normal}.

It is easy to verify (iii). Indeed, $Kz$, regarded as a left $H$-module, is trivial, whence the dual
right $H$-module $V/W$ is trivial.  

For (iv), let $v \in V$, $w \in W$. Since $\mathrm{Lie}(H)$ consists of all elements of $\mathrm{Lie}(G)$ 
that annihilate $g$, it suffices to prove that $[v,w]$ annihilates $z$, or $[v,w]\triangleright z=0$ 
with respect to the $\mathrm{Lie}(G)$-action arising from the left $G$-action on $V^*$. 
For this, using the canonical pairing $V \times V^* \to K$, we wish to prove 
\[ \langle v'\triangleleft [v,w], z \rangle =0, \quad  v' \in V  .\]
Here one should note that the left-hand side equals $\langle v', [v,w]\triangleright z \rangle$.  
If $v'\in W$, the equality holds since $W$ is $\mathrm{Lie}(G)$-stable. If $v \in W$, the equality holds since
we have 
\[ v'\triangleleft [v,w] = -v \triangleleft [w,v']- w \triangleleft [v,v']\in W. \]
If $v, v' \in V \setminus W$, we may assume $v=v'$, due to the preceding results and the fact
$\dim V/W=1$. Then the desired equality follows since we see from \cite[Definition 6, (d)]{mas3}
\[ v'\triangleleft [v,w] = v \triangleleft [v,w] = -\frac{1}{2}w\triangleleft [v,v]\in W. \]
This completes the proof of Claim. 

The quotient supergroup $\mathbb{G}/\mathbb{H}$ corresponds to a Harish-Chandra pair of the form $(D, V/W)$,
where $K[D]=K[g^{\pm 1}]$, $\dim V/W=1$, and the $D$-module structure on $V/W$ is given by $g$.
We conclude from Example \ref{ex:Ggx} (4) that $\mathbb{G}/\mathbb{H}$ appears in 
\eqref{eq:gen_by_1skewprim}. 
\end{proof}

\begin{theorem}\label{thm:even-trigonalizable}
Every even-trigonalizable supergroup $\mathbb{G}$ has a normal chain of closed super-subgroups
\[ \mathbb{G}_0 \lhd \mathbb{G}_1\lhd \ldots\lhd \mathbb{G}_t=\mathbb{G}, \quad t \ge 0 \]
such that $\mathbb{G}_0$ is a trigonalizable algebraic group, and each factor
$\mathbb{G}_i/\mathbb{G}_{i-1}$, $0<i\le t$, is isomorphic to 
$\mathsf{G}_a^-$, $\mathsf{G}_m$ or $\mu_n$ for some $n>1$. 
\end{theorem}
\begin{proof}
We prove by induction on $d=\dim \mathrm{Lie}(\mathbb{G})_1$. 
If $d=0$, we have the result with $t=0$. 
Suppose $d>0$.
Then Lemma \ref{lem:quotient} gives a normal closed super-subgroup $\mathbb{H}$ such that
$\mathbb{G}/\mathbb{H}$ appears in \eqref{eq:gen_by_1skewprim}. 
We have $\mathbb{H}\lhd \mathbb{N} \lhd \mathbb{G}$ such that
$\mathbb{N}/\mathbb{H}\simeq \mathsf{G}_a^-$, and $\mathbb{G}/\mathbb{N}$ is trivial, 
or isomorphic to $\mathsf{G}_m$ or $\mu_n$. 
Since $\dim \mathrm{Lie}(\mathbb{G}/\mathbb{H})_1=1$, we have $\dim \mathrm{Lie}(\mathbb{H})_1=d-1$
by Lemma \ref{lem:quotientHCP}. 
Since $\mathbb{H}_{ev}$ is a closed subgroup
of the trigonalizable $\mathbb{G}_{ev}$, $\mathbb{H}_{ev}$ is trigonalizable, or 
$\mathbb{H}$ is even-trigonalizable. The induction hypothesis applied
to $\mathbb{H}$ proves the theorem. 
\end{proof}

Given an algebraic supergroup $\mathbb{G}$, the \emph{derived super-subgroup} 
$\mathcal{D}\mathbb{G}$ of $\mathbb{G}$ 
is the smallest closed normal super-subgroup $\mathbb{N}$ such that $\mathbb{G}/\mathbb{N}$ is abelian.
The construction of this $\mathcal{D}\mathbb{G}$ is essentially the same as the one given in
\cite[Page 73]{water} in the non-super situation, and it, therefore, commutes with 
extension of the base field. The Hopf super-subalgebra
$K[\mathbb{G}/\mathcal{D}\mathbb{G}]$ of $K[\mathbb{G}]$ is characterized as the 
largest super-cocommutative super-subcoalgebra; this also shows the commutativity with base extension. 
We define inductively,
$\mathcal{D}^0\mathbb{G} = \mathbb{G}$, $\mathcal{D}^r\mathbb{G} 
= \mathcal{D}(\mathcal{D}^{r-1}\mathbb{G})$, $r >0$. We say that $\mathbb{G}$
is \emph{solvable} if $\mathcal{D}^r\mathbb{G}=1$ for some $r$. This is equivalent to 
that any/some base extension of $\mathbb{G}$ is solvable.

\begin{corollary}\label{cor:even-trigonalizable_solvable}
Every even-trigonalizable supergroup is solvable.
\end{corollary}
\begin{proof}
Recall that every trigonalizable algebraic group is solvable.
This, applied to the
$\mathbb{G}_0$ in a normal chain as in Theorem \ref{thm:even-trigonalizable}, proves the corollary.  
\end{proof}
\begin{corollary}\label{cor:converse} 
Let $\mathbb{G}$ be a connected smooth supergroup. Then the following are equivalent:
\begin{itemize}
\item[(a)] $\mathbb{G}$ is solvable;
\item[(b)] $\mathbb{G}_{ev}$ is solvable.
\end{itemize} 
\end{corollary}
\begin{proof}
By base extension we may suppose that $K$ is algebraically closed.
In this case we prove that Conditions (a), (b) and the following are all equivalent. 
\begin{itemize}
\item[(c)] $\mathbb{G}$ is even-trigonalizable.
\end{itemize}

Obviously, (a) $\Rightarrow$ (b). Corollary \ref{cor:even-trigonalizable_solvable} proves (c) $\Rightarrow$ (a).

The Lie-Kolchin Triangularization Theorem \cite[Theorem 10.2]{water} tells us
that a connected smooth algebraic group over an algebraically close field is trigonalizable, 
provided it is solvable. This proves (b) $\Rightarrow$ (c), since $\mathbb{G}_{ev}$ is now
smooth (and connected) by Proposition \ref{prop:smooth_Hopf} in the Appendix.
\end{proof}

Ulyashev and the second-named author \cite[Theorem 4.2]{zubul} proved the result above, 
when $\mathrm{char}\, K=0$, in which case every algebraic supergroup is smooth; 
see Proposition \ref{prop:smooth_Hopf}, again. 

\section{Nilpotency criteria for connected supergroups}\label{sec:nilpotent_supergroup}

We aim to prove our second main result, Theorem \ref{thm:nilpotency}.
 
\subsection{}\label{subsec:center}
Let $\mathbb{G}$ be an affine supergroup. A closed super-subgroup $\mathbb{H}$ is said to be
\emph{central}, if for every $R\in \mathsf{SAlg}_K$, the subgroup $\mathbb{H}(R)$ of $\mathbb{G}(R)$
is central. The condition is equivalent to that the right, say, adjoint $\mathbb{H}$-action on 
$\mathbb{G}$ is trivial. 

Just as in the non-super situation \cite[Page 27]{water}, 
the \emph{center} $\mathcal{Z}(\mathbb{G})$ of $\mathbb{G}$ is defined to be the group-valued functor 
such that $\mathcal{Z}(\mathbb{G})(R)$ consists of those elements of $\mathbb{G}(R)$
whose natural images under every morphism $R \to S$ in $\mathsf{SAlg}_K$ are central in $\mathbb{G}(S)$. 
It is proved in \cite[Section 1]{zubgrish} 
that $\mathcal{Z}(\mathbb{G})$ is representable and is indeed a closed super-subgroup of $\mathbb{G}$; 
the center is thus the largest central closed super-subgroup of $\mathbb{G}$.  

\begin{prop}\label{prop:center_ev}
Let $\mathbb{G}$ be an algebraic supergroup. Let 
\begin{equation}\label{eq:rho}
\rho : \mathbb{G}_{ev} \to \mathsf{GL}(V)
\end{equation} 
denote the linear representation associated with the right $\mathbb{G}_{ev}$-module 
$V= \mathrm{Lie}(\mathbb{G})_1$
defined by \eqref{eq:analogous_adjoint}. Then we have
\[ \mathcal{Z}(\mathbb{G})_{ev} = \mathcal{Z}(\mathbb{G}_{ev})\cap \mathrm{Ker}\, \rho. \]
\end{prop}

After an earlier version of this paper was submitted, Proposition \ref{prop:center_ev} was generalized
by Corollary 5.10 of \cite{masshiba2}; it describes the Harish-Chandra pair corresponding 
to $\mathcal{Z}(\mathbb{G})$, which thus contains the description of $\mathcal{Z}(\mathbb{G})_{ev}$
above. But we retain below our proof of the proposition, 
whose method is different from the one used in \cite{masshiba2}. 

\begin{proof}
Set $G=\mathbb{G}_{ev}$. 
Suppose that $\lambda : K[\mathbb{G}]\to K[\mathbb{G}]\otimes  K[G]$ represents
the right $G$-adjoint action on $\mathbb{G}$. Note that this is dualized to 
the $G$-module structure on $V$. Since 
$\mathcal{Z}(\mathbb{G})_{ev}= \mathcal{Z}(\mathbb{G})\cap G$,
the inclusion $\subset$ follows. 

To see $\supset$, set $H= \mathcal{Z}(G)\cap \mathrm{Ker}\, \rho$. We have to prove
that for every $\gamma \in H(R)$, where $R\ne 0$ is an arbitrary commutative algebra, the automorphism 
\begin{equation}\label{eq:automorphism}
(\mathrm{id}_{K[\mathbb{G}]} \otimes \gamma)\circ \lambda : 
K[\mathbb{G}]\otimes R \to K[\mathbb{G}]\otimes R 
\end{equation}
of the Hopf superalgebra $K[\mathbb{G}]\otimes R$ over $R$ is the identity map. To see that the corresponding
automorphism of the algebraic supergroup $\mathbb{G}_R$ over $R$ is the identity map, 
it suffices by \cite[Therem 4.23]{masshiba}, cited in Remark \ref{rem:equivalence_generalized},
to prove the corresponding automorphism, say $(\phi, \psi)$, of the Harish-Chandra pair $(G_R,V \otimes R)$ 
over $R$ is the identity map, but this is easy to see. Indeed, 
$\gamma \in \mathcal{Z}(G)$ implies $\phi=\mathrm{id}$, while
$\gamma \in \mathrm{Ker}\, \rho$ implies $\psi= \mathrm{id}$. 

As an additional remark, it is easy to see that the assumptions 
required by \cite[Therem 4.23]{masshiba} are now satisfied; for example, 
the base ring $R$ is $2$-torsion free since $2^{-1} \in K \subset R$. 
\end{proof}

\subsection{}\label{subsec:m-part}
Recall that an algebraic group $F$ is said to be \emph{of multiplicative type}, if 
it is diagonalizable after base extension to the algebraic closure of $K$. 

The next lemma follows from Proposition \ref{prop:abelian_supergroup}
and the corresponding
result for abelian algebraic groups. 

\begin{lemma}\label{lem:m-part}
Every abelian algebraic supergroup $\mathbb{H}$ includes uniquely a closed
super-subgroup $F$ such that (i)~$F$ is an algebraic group of multiplicative type, and 
(ii)~$\mathbb{H}/F$ is unipotent. 
If $K$ is perfect, then the embedding $F \hookrightarrow \mathbb{H}$ uniquely splits. 
\end{lemma}

We let $\mathbb{H}_m$ denote the $F$ above. 

\subsection{}\label{subsec:nilpotent_connected_supergroup}
Let $\mathbb{G}$ be an algebraic supergroup. We define normal closed super-subgroups 
$\mathcal{Z}^1\mathbb{G} \subset \mathcal{Z}^2\mathbb{G} \subset \dots$ of $\mathbb{G}$, inductively by 
$\mathcal{Z}^1\mathbb{G} = \mathcal{Z}(\mathbb{G})$,\ $\mathcal{Z}^r\mathbb{G}/\mathcal{Z}^{r-1}\mathbb{G} 
= \mathcal{Z}(\mathbb{G}/\mathcal{Z}^{r-1}\mathbb{G})$, $r >1$. We say that $\mathbb{G}$
is \emph{nilpotent} if $\mathcal{Z}^r\mathbb{G}=\mathbb{G}$ for some $r$. The smallest $r>0$ such that
$\mathcal{Z}^r\mathbb{G}=\mathbb{G}$ is called the \emph{nilpotency length} of $\mathbb{G}$. 

Nilpotent supergroups are solvable. 

The following was proved by the second-named author \cite{zub3}. 
We give below an alternative, Hopf-algebraic proof.

\begin{prop}[\text{\cite[Proposition 3.2]{zub3}}]\label{prop:unipotent_nilpotent}
Every unipotent algebraic supergroup is nilpotent. 
\end{prop}
\begin{proof}
Let $\mathbb{G}$ be a unipotent algebraic supergroup, and set $A = K[\mathbb{G}]$. We may assume
$\mathbb{G}\ne 1$, and so $P(A) \ne 0$; see \cite[Section 10.0]{sw}.
Recall from Section \ref{subsec:supergroups} that the right $\mathbb{G}$-adjoint action on $\mathbb{G}$ 
makes $A$ into a Hopf-algebra object in $\mathbb{G}\text{-}\mathsf{SMod}$. 
Note that $P(A)$ is $\mathbb{G}$-stable, and is indeed a trivial $\mathbb{G}$-supermodule. 
Hence it generates a Hopf super-subalgebra $B_1 \,(\ne K)$
of $A$ which is trivial as a $\mathbb{G}$-supermodule; we remark that as this $B_1$, a super-cocommutative
Hopf super-subalgebra of $A$ is a trivial $\mathbb{G}$-supermodule.
If $B_1 \subsetneqq A$, we have 
the quotient Hopf superalgebra $A/B_1^+A\, (\ne K)$ (see \eqref{eq:quotient_Hopf}), 
which is indeed a Hopf-algebra object in $\mathbb{G}\text{-}\mathsf{SMod}$, again. We see that
$P(A/B_1^+A)\, (\ne 0)$ is $\mathbb{G}$-stable, and includes a non-zero trivial $\mathbb{G}$-super-submodule 
since $\mathbb{G}$ is unipotent. It generates a Hopf super-subalgebra $(\ne K)$ of  $A/B_1^+A$  
which is trivial as a $\mathbb{G}$-supermodule.
We thus have a chain,
$K=B_0 \subsetneqq B_1 \subsetneqq B_2$, of $\mathbb{G}$-stable Hopf super-subalgebras of $A$,
such that each quotient Hopf superalgebra $B_i/B_{i-1}^+B_i$ is trivial as a $\mathbb{G}$-supermodule.
Continue the process so long as $B_r \subsetneqq A$.
But it must end after finitely many steps, or $B_r =A$ for some $r$, since
the ascending chain $B_1^+A \subsetneqq B^+_2A \subsetneqq \dots$ of Hopf super-ideals of the Noetherian $A$ becomes 
stationary; see Section \ref{subsec:smooth_superalgebras} in the Appendix. 
We have a chain, $1=\mathbb{N}_r \subsetneqq \dots \subsetneqq \mathbb{N}_1
\subsetneqq \mathbb{N}_0 = \mathbb{G}$, of normal closed super-subgroups of $\mathbb{G}$, such that
$K[\mathbb{G}/\mathbb{N}_i] = B_i$, and each $\mathbb{N}_{i-1}/\mathbb{N}_i$ is central in 
$\mathbb{G}/\mathbb{N}_i$. This proves the desired nilpotency. 
\end{proof}

\begin{theorem}\label{thm:nilpotency}
Let $\mathbb{G}$ be an algebraic supergroup. Concerning the conditions given below we have
(a) $\Rightarrow$ (b) $\Rightarrow$ (c) $\Rightarrow$ (d). If in addition, $\mathbb{G}$ is 
connected, then these conditions are equivalent to each other. 
\begin{itemize}
\item[(a)] $\mathbb{G}/\mathcal{Z}(\mathbb{G})_m$ is unipotent;
\item[(b)] $\mathbb{G}$ fits into a central extension $1 \to F \to \mathbb{G} \to \mathbb{U} \to 1$ 
of a unipotent supergroup $\mathbb{U}$ by an algebraic group $F$ of multiplicative type; 
\item[(c)] $\mathbb{G}$ is nilpotent; 
\item[(d)] $\mathbb{G}_{ev}$ is nilpotent, and $\mathcal{Z}(\mathbb{G}_{ev})_m \subset \mathrm{Ker}\, \rho$,
where $\rho$ is the linear representation \eqref{eq:rho}. 
\end{itemize}
\end{theorem}
\begin{proof}
(a) $\Rightarrow$ (b). This is obvious.\ 

(b) $\Rightarrow$ (c).\  This follows from Proposition \ref{prop:unipotent_nilpotent}. 

(c) $\Rightarrow$ (d).\ Let $(H,W)$ and $(G,V)$ be the Harsh-Chandra pair assigned to
$\mathcal{Z}(\mathbb{G})$ and $\mathbb{G}$, respectively. By Lemma \ref{lem:quotientHCP} we have the
Harish-Chandra pair $(G/H, V/W)$ assigned to $\mathbb{G}/\mathcal{Z}(\mathbb{G})$.
We remark that $W$ is trivial as a $G\, (=\mathbb{G}_{ev})$-module, and hence as a $\mathcal{Z}(G)_m$-module.  

Assume (c). Obviously, $G$ is nilpotent. We wish to prove by induction on the nilpotency length of $\mathbb{G}$
that $V$ is trivial as a $\mathcal{Z}(G)_m$-module. The induction hypothesis 
applied to $\mathbb{G}/\mathcal{Z}(\mathbb{G})$ shows that $V/W$ is trivial as a $\mathcal{Z}(G/H)_m$-module.
This, combined with the remark above, 
implies the desired result, since $\mathcal{Z}(G)_m$ is linearly reductive, and
the quotient morphism 
$G\to G/H$ induces $\mathcal{Z}(G)_m \to \mathcal{Z}(G/H)_m$. 

It remains to prove (d) $\Rightarrow$ (a), assuming that $\mathbb{G}$ is connected. 
Set $G=\mathbb{G}_{ev}$. 
Assume (d). The second condition, combined with Proposition \ref{prop:center_ev}, implies
\[ \big(\mathbb{G}/\mathcal{Z}(\mathbb{G})_m \big)_{ev}=G/\mathcal{Z}(\mathbb{G})_m
= G/\mathcal{Z}(G)_m. \]
Due to Proposition \ref{prop:unipotency} our aim is to 
prove that $G/\mathcal{Z}(G)_m$
is unipotent. Theorem 1.10 of \cite[IV, \S 4]{dg}, 
applied to the connected and nilpotent $G$, shows this desired result. 
\end{proof}

\begin{example}\label{ex:nilpotent_Ggx}
The algebraic supergroup $G_{g,x}$ constructed in Section \ref{sec:Ggx} is nilpotent
if and only if $G$ is nilpotent, and $g=1$.
To prove this we may suppose that $G$ is nilpotent, and proceed by induction on the nilpotency length of $G$. 
For ``only if" suppose that $G_{g,x}$ is nilpotent. By 
(c) $\Rightarrow$ (d), proved above, the natural image of $g$ in $K[\mathcal{Z}(G)]$ is $1$. 
It follows that 
$\mathcal{Z}(G)$ is central in $G_{g,x}$, and $g \in K[G/\mathcal{Z}(G)]$. 
The induction hypothesis applied to 
$G_{g,x}/\mathcal{Z}(G)=(G/\mathcal{Z}(G))_{g,x}$ shows $g=1$. 
The same argument, with $g$ replaced by $1$, shows ``if". 
\end{example}

\section{Some counter-examples}\label{sec:counter-examples}

\subsection{}\label{subsec:counter-example1}
Let $\mathbb{G}$ be an affine supergroup. 
Recall that $\mathbb{G}_u$ denotes
the unipotent radical of $\mathbb{G}$. 
Let us discuss when the quotient morphism $\mathbb{G} \to \mathbb{G}/\mathbb{G}_u$ splits. 
If this is the case, then $\mathbb{G}$ decomposes into a semi-direct product, 
\[ \mathbb{G} \simeq \mathbb{G}/\mathbb{G}_u \ltimes \mathbb{G}_u, \]
as is easily seen. Of course the classical Chevalley decomposition for affine groups is in our mind. 
A weak form of this classical
result will be proved in the generalized super context, in the next section; see Proposition \ref{prop:split}.  
We pose here a somewhat more ambitious question; cf. the cited Proposition in Case (b).

\begin{Qu}\label{qu:possibly_generalized}
Given a super-trigonalizable supergroup $\mathbb{G}$ (see Definition \ref{def:even-diagonalizable} (3)), 
does the quotient morphism $\mathbb{G}\to \mathbb{G}/\mathbb{G}_u$ split? 
\end{Qu}

However, the answer is negative, as will be seen from the example below.

Let $G=\mathsf{G}_a \times \mathsf{G}_m$. Then $\mathrm{Lie}(G)$ is spanned by two non-zero elements $x$, $y$ which
span $\mathrm{Lie}(\mathsf{G}_a)$ and $\mathrm{Lie}(\mathsf{G}_m)$, respectively. 
Given $\alpha, \beta \in K$, let
\[ \mathbb{G} =G_{1,{\alpha x+\beta y}}. \]
Recall from Lemma-Definition \ref{lemdef:Ggx} that this is the algebraic supergroup 
corresponding to the Harish-Chandra pair $(G,V)$ which consists of $G$, a one-dimensional 
trivial right $G$-module $V=Kv$, and the bracket determined by $[v,v]=2(\alpha x+ \beta y)$.   
Note that this $\mathbb{G}$ includes $\mathsf{G}_a= \mathsf{G}_a \times 1$ as a central closed
super-subgroup (see Example \ref{ex:nilpotent_Ggx}), and
\[ \mathbb{G}/\mathsf{G}_a \simeq (\mathsf{G}_m)_{1,\beta y}. \]

\begin{lemma}\label{lem:counterexample}
We have the following.
\begin{itemize}
\item[(1)] $\mathbb{G} \to \mathbb{G}/\mathsf{G}_a$ splits if and only if $\alpha=0$.
\item[(2)] $\mathbb{G}_u = \mathsf{G}_a$, if and only if $(\mathsf{G}_m)_{1,\beta y}$ is super-diagonalizable, 
if and only if $\beta\ne 0$.
\end{itemize}
\end{lemma}
\begin{proof} 
(1)\ The two conditions both are equivalent to that the embedding $\mathsf{G}_m = 1 \times \mathsf{G}_m \to G$
and $\mathrm{id}_V$ give a morphism $(\mathsf{G}_m, V) \to (G, V)$ in $\mathsf{HCP}$, where 
in $(\mathsf{G}_m, V)$, $V$ is a trivial right $\mathsf{G}_m$-module and the bracket is determined by 
$[v,v] = 2\beta y$. 

(2)\ The three conditions all are equivalent to that the unipotent radical of  $(\mathsf{G}_m)_{1,\beta y}$ 
is trivial; see Corollary \ref{cor:Dgx_representation} (2).
\end{proof}

We conclude that if $\alpha \beta \ne 0$, then the algebraic supergroup
$\mathbb{G}$ is super-trigonalizable, but $\mathbb{G} \to \mathbb{G}/\mathbb{G}_u$ does not split. 

\subsection{}\label{subsec:counter-example2}
On the other hand one may ask the following.

\begin{Qu}\label{qu:super-diagonalizable}
Is every super-diagonalizable supergroup isomorphic to a direct product of a diagonalizable algebraic group and 
various $D_{g,x}$?
\end{Qu}

The answer is negative again, as will be seen from the following example. 

Let $y$ be an element which spans $\mathrm{Lie}(\mathsf{G}_m)$, as above. 
Let $V\ne 0$ be a trivial $\mathsf{G}_m$-module and $[\ , \ ] : V \times V \to Ky\, (=K)$ 
be a non-degenerate symmetric 
bilinear form. Then $(\mathsf{G}_m, V)$ is a Harish-Chandra pair.
The corresponding algebraic supergroup, say $\mathbb{G}$, is super-diagonalizable,
as is seen from Proposition \ref{prop:characterize_super-diagonalizable}. 

Note that $\mathsf{G}_m$ cannot decompose non-trivially 
into a direct product. The non-degeneracy shows that if $V=V_1\oplus V_2$, $V_1\ne 0$, $V_2\ne 0$ and $[V_1, V_2]=0$,
then $[V_1,V_1]= [V_2,V_2]=Ky$, whence $[V_1,V_1]\cap [V_2,V_2]\ne 0$. 
It follows that the Harish-Chandra pair $(\mathsf{G}_m, V)$
cannot decompose non-trivially into a direct product (in the obvious sense), and so $\mathbb{G}$ cannot, either. 
Therefore, this $\mathbb{G}$ gives a negative answer to the question when $\dim V > 1$. 

\subsection{}\label{sebsec:problem}
The examples given in the last two subsections motivate us to
pose the following. 

\begin{problem}\label{problem}
(1)\ Characterize those super-trigonalizable supergroups for which the quotient morphism 
$\mathbb{G} \to \mathbb{G}/\mathbb{G}_u$ split. 

(2)\ Characterize those super-diagonalizable supergroup which are isomorphic to a direct product
of diagonalizable algebraic group with various $D_{g,x}$. 
\end{problem}

\section{Super-analogue of the Chevalley Decomposition Theorem}\label{sec:splitting_property}

As was announced at the beginning of Section \ref{subsec:counter-example1} we prove a weak form of
the classical Chevalley Decomposition Theorem in the super context. 

\begin{prop}\label{prop:split}
Let $\mathbb{G}$ be an affine supergroup. The quotient morphism $\mathbb{G} \to \mathbb{G}/\mathbb{G}_u$ splits if
\begin{itemize}
\item[(a)] (i)~$\mathrm{char}\, K=0$, and (ii)~$\mathbb{G}/\mathbb{G}_u$ is linearly reductive, or
\item[(b)] (i)~$K$ is an algebraically closed field of $\mathrm{char}\, K > 2$, and 
(ii)~$\mathbb{G}/\mathbb{G}_u$ is purely even and diagonalizable. 
\end{itemize}
\end{prop}

\begin{rem}\label{rem:split}
Our result in Case (a) is indeed weaker than the classical one,  
in that we have to assume (ii) even under (i), while the classical result on affine groups
proves that (ii) holds under (i). The assumption is indeed necessary since we have many examples
of supergroups, such as the $D_{g,x}$ with $x\ne0$ (see Corollary \ref{cor:Dgx_representation}), 
which have trivial unipotent radical, but are not linearly reductive. 
\end{rem}

\begin{proof}[Proof of Proposition \ref{prop:split} in Case (b)] In this case the result is a direct
consequence of the classical one, as will be seen below. 

Set $C = \mathrm{Corad}(K[\mathbb{G}])$. By Lemma \ref{lem:unipotent_radical} Condition (ii) is equivalent to 
$C$ is spanned by grouplikes, so that it is necessarily a purely even Hopf super-subalgebra of $K[\mathbb{G}]$.
By \cite[Proposition 4.6 (3)]{mas2}
the composite of the inclusion $C \subset K[\mathbb{G}]$ with the
quotient map onto $K[\mathbb{G}_{ev}]$ is an injection, through which $C$ maps onto
$\mathrm{Corad}(K[\mathbb{G}_{ev}])$. Under (i), the classical result 
(see \cite[Theorem 3.3]{mas1}, for example) gives a splitting, say $\pi$, of the injection.
The composite $K[\mathbb{G}] \to K[\mathbb{G}_{ev}] \to C$ of the quotient map with $\pi$ gives a
desired splitting $\mathbb{G}/\mathbb{G}_u \to \mathbb{G}$. 
\end{proof}

A Hopf-algebraic proof of the classical Chevalley Decomposition Theorem is given in \cite{mas1}.  
Our proof of Proposition \ref{prop:split} in Case (a), which starts with preparing the following lemma, 
modifies the cited proof so as to fit in with the super context.  

\begin{lemma}\label{lem:split}
Let $C$ be a Hopf superalgebra which is not necessarily super-commutative. Assume that $C$ is cosemisimple, 
or $C = \mathrm{Corad}(C)$.
\begin{itemize}
\item[(1)] The unit map $u : K \to C,\ u(1) = 1$ splits as a left (or right) $C$-super-comodule map. 
\item[(2)] The coproduct $\Delta : C \to C \otimes C$ of $C$ splits as a $(C, C)$-super-bicomodule map.
\item[(3)] Let $C \subset Z$ is an inclusion
of left $C$-supermodule coalgebras. If $C = \mathrm{Corad}(Z)$, then the inclusion 
$C \hookrightarrow Z$ splits as a left $C$-supermodule coalgebra map. 
\end{itemize} 
\end{lemma}

Here, a \emph{left $C$-supermodule coalgebra} is a coalgebra-object in the abelian tensor category
of $C\text{-}\mathsf{SMod}$ of left $C$-supermodules. A \emph{left $C$-supermodule coalgebra map}, 
which is by definition a coalgebra morphism in the category, is precisely a left $C$-linear super-coalgebra map. 

\begin{proof}
(1) This follows since the category of left $C$-super-comodules is semisimple, and $u$ is a morphism
in that category. 

(2) This follows from (1), as in the non-super situation.

(3) By (2), $C$ is a coseparable coalgebra in $C\text{-}\mathsf{SMod}$. The assumption 
$C = \mathrm{Corad}(Z)$ implies $Z = \bigcup_n\wedge^n_Z C$; see \cite[Theorem 4.16~(c)]{ams}. 
We can apply the just cited result 
for coseparable coalgebras in an arbitrary abelian tensor category,
to our $C$ in $C\text{-}\mathsf{SMod}$. Then it follows that 
the identity map $C \to C$ extends to a coalgebra morphism $Z \to C$ in $C\text{-}\mathsf{SMod}$.
\end{proof}

\begin{proof}[Proof of Proposition \ref{prop:split} in Case (a)]
Before going into the proof, which consists of four steps, here are two general remarks.

1.\ Given a super-coalgebra $C$ and a superalgebra $R$, all super-linear
maps $C \to R$ form a group, $\mathsf{SMod}_K(C,R)$, with respect to the convolution-product $*$ 
\cite[Page 72]{sw}. By saying that a super-linear map is $*$-\emph{invertible}, we mean that it is invertible
in a relevant group of this sort. 

2.\ A main ingredient of the proof in \cite{mas1} is bi-crossed products. The construction of bi-crossed 
products is directly generalized to our super context. This is based on the fact that Doi and Takeuhi's
results in \cite{dt} on cleft comodule algebras are generalized to the super context, 
or indeed more generally, to the context of braided
category; see \cite{gg}, for example.

\emph{Step 1.}\quad Set 
\[ A = K[\mathbb{G}],\quad C =\mathrm{Corad}(K[\mathbb{G}]). \]
(This $C$ is denoted by $K$ in \cite{mas1}.)
Until the end of Step 3 we only assume (ii), or equivalently, that $C$ is a 
Hopf super-subalgebra of $A$; see Lemma \ref{lem:unipotent_radical}. 
As in \cite[Page 115, lines 17--19]{mas1}, we consider all pairs $(B, \varpi)$ of 
a Hopf super-subalgebra $B\subset A$ including $C$ and a Hopf superalgebra map $\varpi : B \to C$
such that $\varpi|_C =\mathrm{id}_C$, and introduce the obvious order into the set of the pairs. 
By Zorn's Lemma we have a maximal pair $(B,\varpi)$. It suffices to prove $B=A$. 
On the contrary we suppose $B \subsetneqq A$ for a contradiction.

\emph{Step 2.}\quad Let $H=A/B^+A$ be the quotient Hopf superalgebra of $A$, as in \eqref{eq:quotient_Hopf}.
Since $B \subsetneqq A$, we have $K \subsetneqq H$. Note
\begin{equation}\label{eq:CoradH}
K=\mathrm{Corad}(H).
\end{equation} 
Then it follows that $P(H) \ne 0$.
Since $A$ is injective as a right $H$-super-comodule by \cite[Theorem~5.9 (2)]{mas2}, 
the unit map $K\to A$ extends to a right $H$-super-comodule map $H \to A$, 
which is $*$-invertible by \eqref{eq:CoradH}, and can be chosen as to preserve the counit.  
Hence $A$ is presented as a bi-crossed product,
\begin{equation*}
A = B\, {}^{\tau}\hspace{-1mm}\bcp{}_{\sigma} H, 
\end{equation*}
which is constructed on $B \otimes H$ by the trivial action 
$\rightharpoonup : H\otimes B\to B,\ h \rightharpoonup b=\varepsilon_H(h)b$, 
a co-measuring $\rho : H \to H\otimes B$, a (super-symmetric) cocycle $\sigma : H \otimes H \to B$ and
a dual cocycle $\tau : H \to B \otimes B$. All these are super-linear maps,  
and are supposed to be \emph{normalized}, satisfying the condition given in \cite[Page 112, line--5]{mas1} 
and the dual one;
$\sigma$ and $\tau$ are $*$-invertible. 
By saying that $\rho$ is a \emph{co-measuring}, 
we mean that $(H, \rho)$ is a coalgebra object in $\mathsf{SComod}\text{-}B$.
This implies 
\begin{equation}\label{eq:P(H)}
\rho(P(H)) \subset P(H) \otimes B. 
\end{equation}

\emph{Step 3.}\quad Let $m_R$ denote the product on any superalgebra $R$. 
Let $\iota_B$ and $\iota_H$ denote the natural embeddings $B \to B\otimes H=A$ and $H \to B \otimes H=A$,
respectively. Compose $\iota_H\otimes \iota_H$, first with each side of 
\[ \Delta_A\circ m_A = m_{A\otimes A} \circ (\Delta_A \otimes \Delta_A), \]
and then with $(\varepsilon_B\otimes \mathrm{id}_H)\otimes(\mathrm{id}_B\otimes \varepsilon_H)$. Then we have
the equation 
\[ (\iota_B\circ \sigma) * (\rho\circ m_H) = (m_{H\otimes B}\circ(\rho\otimes \rho)) *(\iota_B\circ \sigma) \]
in the group $\mathsf{SMod}_K(H\otimes H, H\otimes B)$; this is analogous to 
the equation given in \cite[Page 115, line--13]{mas1}. It follows that if $J \subset H$ is 
a super-cocommutative Hopf superalgebra, then $\rho|_J : J \to H \otimes B$ is a superalgebra map. 
Suppose that $J$ is the super-cocommutative Hopf superalgebra generated by 
$P(H)\, (\ne 0)$. Then $J \supsetneqq K$, and $\rho(J) \subset J \otimes B$ by \eqref{eq:P(H)}. 

We see that
\[ B':=  B\, {}^{\tau}\hspace{-1mm}\bcp{}_{\sigma} J \]
is a Hopf superalgebra of $A$ which properly includes $B$. Here the associated cocycle and dual cocycle 
are induced from the $\sigma$ and the $\tau$ before, and are denoted by the same symbols. 
Let $I=(\mathrm{Ker}\, \varpi)$
be the Hopf super-ideal of $A$ generated by the Hopf super-ideal $\mathrm{Ker}\, \varpi$ of $B$. 
Since $B/\mathrm{Ker}\, \varpi =C$, we have
\[ B'/I= C\, {}^{\overline{\tau}}\hspace{-1mm}\bcp{}_{\overline{\sigma}} J, \]
where the associated co-measuring $\overline{\rho}$ as well as $\overline{\sigma}$ and $\overline{\tau}$
are naturally induced from $\rho$, $\sigma$ and $\tau$, respectively. Since $\mathrm{Corad}(B'/I)=C$, we can
apply Lemma \ref{lem:split} (3) for $C \subset Z$ to the present $C \subset B'/I$. Then we see that
$\overline{\rho}$, $\overline{\sigma}$ and $\overline{\tau}$ can be re-chosen so that $\overline{\tau}$ 
is trivial, and so $B'/I= C\, \bcp{}_{\overline{\sigma}} J$ is, as a coalgebra-object in $C\text{-}\mathsf{SMod}$, 
the smash coproduct constructed by the re-chosen $\overline{\rho} : J \to J \otimes C$. 

\emph{Step 4.}\quad Assume (i), or $\mathrm{char}\, K=0$. 
Then $J$ is, as a superalgebra, the tensor product 
$\mathrm{Sym}(P(H)_0)\otimes \wedge(P(H)_1)$ of the symmetric algebra on $P(H)_0$ and the exterior algebra 
on $P(H)_1$. The natural embedding $P(H) \to J \to C\, \bcp{}_{\overline{\sigma}} J=B'/I$ uniquely extends
to a superalgebra map $\phi : J \to B'/I$, which is a right $J$-comodule map since the restriction
$\phi|_{J'}$ to $J':=K \oplus P(H)$ is such. 
Moreover, $\phi$ is $*$-invertible since the restriction to $K=\mathrm{Corad}(J)$ is so.
It follows from the super-analogue of Doi and Takeuchi's Theorem 
\cite[Theorem 9]{dt} (see \cite[Theorem 10.6]{gg}) that 
\[ C \otimes J \to B'/I,\quad c \otimes a \mapsto c \, \phi(a) \] 
is an isomorphism of (right $J$-comodule) superalgebras over $C$. If we regard the domain $C \otimes J$
as the smash-coproduct super-coalgebra $C \cmdblackltimes J$ constructed by $\overline{\rho}$, then 
this last is an isomorphism of Hopf superalgebras, since the restriction to $C\otimes J'$ is
obviously a super-coalgebra map. The Hopf super-subalgebra $B'\subset A$, together with the composite of
$B' \to B'/I = C \cmdblackltimes J$ with 
$\mathrm{id}_C \otimes \varepsilon_J: C \cmdblackltimes J\to C$, gives a pair
which is properly bigger than $(B, \varpi)$; this contradicts the maximality of the last pair,
and completes the proof. 
\end{proof}

\appendix
\section{On smooth superalgebras}\label{appendix}

We continue to work over a field $K$ of characteristic $\ne 2$, unless otherwise stated. 
All superalgebras (over $K$) are assumed to be super-commutative, and 
$\mathsf{SAlg}_K$ denotes the category those superalgebras, as before. 

\subsection{Smooth superalgebras}\label{subsec:smooth_superalgebras}
Given a (purely even) algebra $R$ and a module $M$ over $R$, $\wedge_R(M)$ denotes 
the exterior algebra on $M$. This is graded by $\mathbb{N}=\{ 0,1,\dots \}$, and so 
$\wedge_R(M) \in  \mathsf{SAlg}_K$ with the $\mathbb{N}$-grading modulo $2$. 

Let $A \in \mathsf{SAlg}_K$.

We say that $A$ is \emph{smooth over} $K$, 
if given $B \in \mathsf{SAlg}_K$ and a nilpotent (or equivalently, square-zero)
super-ideal $J \subset B$,
the natural map $\mathsf{SAlg}_K(A,B)\to \mathsf{SAlg}_K(A,B/J)$ is a surjection.
A familiar argument using pull-back shows that the condition is 
equivalent to that in $\mathsf{SAlg}_K$, every surjection 
onto $A$ with nilpotent (or equivalently, square-zero) kernel splits. 

We say that $A\, (=A_0 \oplus A_1)$ is \emph{Noetherian} if 
the following conditions, which are easily seen to be 
equivalent to each other, are satisfied:
\begin{itemize}
\item[(a)] The super-ideals of $A$ satisfy the ACC;
\item[(b)] The ring $A_0$ is Noetherian, and the $A_0$-algebra $A$ is generated by
finitely many odd elements;
\item[(c)] $A_0$ is Noetherian, and the $A_0$-module $A_1$ is finitely generated. 
\end{itemize}

Define 
\[ I_A=(A_1),\quad \overline{A}=A/I_A. \]
Thus $\overline{A}$ is the largest purely even quotient of $A$.  
Given a positive integer $n$, we have
\[ I_A^n= \begin{cases} A_1^{n+1} \oplus A_1^n & n~~\text{odd},\\
A_1^n \oplus A_1^{n+1} & n~~\text{even}, \end{cases} \]
whence 
$I_A^n/I_A^{n+1}= A_1^n/A_1^{n+2}$; this is an $\overline{A}$-module, which is purely even (resp., odd)
if $n$ is even (resp., odd). 
We define a graded algebra over $\overline{A}$ by 
\[ \mathsf{gr}(A) = \bigoplus_{n\ge 0} I_A^n/I_A^{n+1} =\overline{A} \oplus I_A/I_A^2 \oplus \dots . \]
We let
\[ \kappa_A : \wedge_{\overline{A}}(I_A/I_A^2) \to \mathsf{gr}(A) \]
denote the graded $\overline{A}$-algebra map induced by the embedding $I_A/I_A^2 \to \mathsf{gr}(A)$.
Obviously this is a surjection. If the $A_0$-module $A_1$ is generated by $s$ elements, then $A_1^{s+1}=0$,
whence $\wedge_{\overline{A}}^n(I_A/I_A^2) = 0= \mathsf{gr}(A)(n)$ for all $n > s$. This is the case
for some $s$, if $A$ is Noetherian. 

Suppose that $A$ is Noetherian. Following Schmitt \cite[Section 3.3, Page 79]{schmitt} 
we define:

\begin{definition}[T.~Schmitt \cite{schmitt}]\label{def:regular}
$A$ is said to be \emph{regular} if
\begin{itemize}
\item[(1)] the Noetherian ring $\overline{A}$ is regular, 
\item[(2)] the finite generated $\overline{A}$-module $I_A/I_A^2 \, (=A_1/A_1^3)$ is projective, and 
\item[(3)] $\kappa_A$ is an isomorphism. 
\end{itemize}
\end{definition}
See Section \ref{subsec:comparison} below for some remarks on this definition. 

We say that $A$ is \emph{geometrically regular over} $K$,
if for every finite field extension $L/K$, the base extension $A\otimes L$, which is Noetherian,
is regular.

\begin{theorem}\label{thm:smooth}
Let $A$ be a Noetherian superalgebra over a field $K$. Then the following (a)--(d) are equivalent to each other:
\begin{itemize}
\item[(a)] $A$ is smooth over $K$;
\item[(b)] $A$ is regular, and the algebra $\overline{A}$ is smooth over $K$;
\item[(c)] $A$ is geometrically regular over $K$;
\item[(d)]
\begin{itemize}
\item[(i)] The algebra $\overline{A}$ is smooth over $K$, 
\item[(ii)] the $\overline{A}$-module $I_A/I_A^2$ is projective, and 
\item[(iii)] there is an isomorphism
$\wedge_{\overline{A}}(I_A/I_A^2) \overset{\simeq}{\longrightarrow} A$ of superalgebras. 
\end{itemize}
\end{itemize}

If $K$ is a perfect field, the equivalent conditions above are equivalent to
\begin{itemize}
\item[(e)] $A$ is regular. 
\end{itemize}
\end{theorem}
\begin{proof}
First, assume that $K$ is an arbitrary field. 

(a) $\Rightarrow$ (b).\
Assume (a). Given a surjection $B \to C$ of algebras with nilpotent kernel, every
algebra map $\overline{A} \to C$, identified with a map $A \to C$ of superalgebras, can lift to
$\overline{A}\to B$. Therefore, $\overline{A}$ is smooth over $K$, which is necessarily 
regular; see \cite[Corollary 9.3.13]{weibel}. 

It remains to verify (2) and (3) in Definition \ref{def:regular}. 
By the smoothness just proved we have a section 
$i : \overline{A} \to A_0$ of $A_0 \to A_0/A_1^2 = \overline{A}$. Through this $i$ we regard $\overline{A}$
as a subalgebra of $A_0$. Let $f : P \to A_1$ be an $\overline{A}$-linear map whose composite with 
$A_1 \to A_1/A_1^3= I_A/I_A^2$
gives a projective cover of the $\overline{A}$-module $I_A/I_A^2$; $P$ is thus a finitely generated
projective $\overline{A}$-module. 
This $f$ uniquely extends to a map $\widetilde{f} : \wedge_{\overline{A}}(P) \to A$ of superalgebras
over $\overline{A}$. Since this $\widetilde{f}$ preserves the projections onto $\overline{A}$, it 
induces the graded superalgebra map 
\[ \mathsf{gr}(\widetilde{f}) : \wedge_{\overline{A}}(P) \to \mathsf{gr}(A) \]
which is identical in degree $0$, 
\begin{equation}\label{eq:degree_zero}
\mathsf{gr}(\widetilde{f})(0)=\mathrm{id}_{\overline{A}}. 
\end{equation}
Since this is a surjection, $\widetilde{f}$ is, too. 
One sees from \eqref{eq:degree_zero} that the kernel
$\mathrm{Ker}(\widetilde{f})$ of $\widetilde{f}$ is nilpotent. Therefore, $\widetilde{f}$ has a section
$h : A \to \wedge_{\overline{A}}(P)$, whence $\mathsf{gr}(h) : \mathsf{gr}(A)\to 
\wedge_{\overline{A}}(P)$ is a section of $\mathsf{gr}(\widetilde{f})$, such that 
$\mathsf{gr}(h)(0)=\mathrm{id}_{\overline{A}}$. 
Note that 
$\mathsf{gr}(\widetilde{f})(1) : P \to I_A/I_A^2$ coincides with the projective cover above,
and has $\mathsf{gr}(h)(1)$ as an $\overline{A}$-linear section. Therefore,
$\mathsf{gr}(\widetilde{f})(1)$ and $\mathsf{gr}(h)(1)$ are inverses to each other. This implies (2) 
and that $\mathsf{gr}(\widetilde{f})$ is identified with $\kappa_A$. 
Since the $\overline{A}$-algebra $\mathsf{gr}(A)$
is generated by the first component,
$\mathsf{gr}(h)$ is a surjection. This implies that $\mathsf{gr}(\widetilde{f})\, (=\kappa_A)$ is an isomorphism,
proving (3). 

(b) $\Rightarrow$ (d).\
Assume (b). Then we have (i) and (ii) of (d). It remains to prove (iii). 

By (i) the algebra surjection $A_0 \to A_0/A_1^2 = \overline{A}$ again has a splitting
$i$, through which we again regard $\overline{A}$ as a subalgebra of $A_0$. 
By (ii) the $\overline{A}$-linear surjection $(I_A)_1=A_1 \to A_1/A_1^3=I_A/I_A^2$ has a section, say $j$.
This $j$ uniquely extends to a map 
\[ \widetilde{j} : \wedge_{\overline{A}}(I_A/I_A^2)\to A \]
of $\overline{A}$-superalgebras.
Just as the $\widetilde{f}$ above, this $\widetilde{j}$ induces a graded superalgebra map
\[ \mathsf{gr}(\widetilde{j}) : \wedge_{\overline{A}}(I_A/I_A^2) \to \mathsf{gr}(A)  \]
such that $\mathsf{gr}(\widetilde{j})(0) =\mathrm{id}_{\overline{A}}$. 
Since one sees that $\mathsf{gr}(\widetilde{j})(1)$ is the identity on $I_A/I_A^2$, 
it follows that $\mathsf{gr}(\widetilde{j})$ coincides with $\kappa_A$, and is an isomorphism 
by (b). Hence $\widetilde{j}$ is the desired isomorphism.

(d) $\Rightarrow$ (a).\  Assume (d).   
By (ii) and (iii) we have: 
(iv)~$A_0$ includes a subalgebra $R$ which is isomorphic to $\overline{A}$,\   
(v)~$A_1$ includes an $R$-submodule $P$ which is finitely generated projective, and 
(vi)~the embedding 
$P \to A$ extends to an isomorphism $\wedge_R(P) \overset{\simeq}{\longrightarrow} A$ of 
$R$-superalgebras.

To show (a) we wish to prove that a surjection 
$q : B \to A$ in $\mathsf{SAlg}_K$ with nilpotent kernel splits. 
Since $q_0$ restricts to the surjection $q_0^{-1}(R) \to R$ of algebras with nilpotent
kernel, it splits by (i). Hence we may suppose $R \subset B_0$ is a subalgebra, and $q$ is an 
$R$-superalgebra map. Since $q_1$ restricts to the $R$-linear surjection 
$q_1^{-1}(P) \to P$ onto the projective module, 
it has a section. 
This section uniquely extends to a map $A = \wedge_{R}(P)\to B$ of $R$-superalgebras, 
which is the desired section.

(b) $\Leftrightarrow$ (c). Let $L/K$ be a field extension. Then
$I_{A\otimes L}/I_{A\otimes L}^2=(I_A/I_A^2)\otimes L$ is projective over 
$\overline{A\otimes L}=\overline{A}\otimes L$ if (and only if) $I_A/I_A^2$ is $\overline{A}$-projective. 
Since $\kappa_{A \otimes L}$ is identified with
the base extension $\kappa_A \otimes \mathrm{id}_L$ of $\kappa_A$ to $L$, 
it follows that $\kappa_{A \otimes L}$ is an isomorphism if (and only if) $\kappa_A$ is. 
Therefore, (c) is equivalent to
\begin{itemize}
\item[(c)$'$] $A$ is regular,
and $\overline{A}$ is geometrically regular over $K$.
\end{itemize} 
The desired equivalence follows from  
the known one 
\cite[Chapter 0, Theorem 22.5.8]{EGA}: 
$\overline{A}$ is smooth 
over $K$ $\Leftrightarrow$ $\overline{A}$ is geometrically regular over $K$. 

To complete the proof assume that $K$ is perfect. 
Then $\overline{A}$ is regular if and only if it is
geometrically regular over $K$; see \cite[(28.N), Page 208]{mats}, for example. 
Thus we have (e) $\Leftrightarrow$ (c)$'$ $\Leftrightarrow$ (c). 
\end{proof}

\subsection{Smooth Hopf superalgebras}\label{subsec:smooth_Hopf}
Suppose that $A$ is a (super-commutative) 
Hopf superalgebra over $K$ which is not necessarily finitely generated.
Define $W^A=A_1/A_0^+A_1$, where $A_0^+=A_0 \cap A^+$. If $\mathbb{G}$ denotes the
affine supergroup represented by $A$, then 
$\overline{A} =A/(A_1)$ is the Hopf algebra which represents $\mathbb{G}_{ev}$ 
(see Section \ref{subsec:intro1}), and $W^A$ is the odd component of the cotangent 
space of $\mathbb{G}$ at $1$. This $W^A$ and the odd component $\mathrm{Lie}(\mathbb{G})_1$
of the Lie superalgebra of $\mathbb{G}$ are dual to each other, 
if $\mathbb{G}$ is an algebraic supergroup. 
By \cite[Theorem 4.5]{mas1} we
have a (counit-preserving) isomorphism of (left $\overline{A}$-comodule) superalgebras
\begin{equation}\label{eq:decomposition}
A \simeq \overline{A} \otimes \wedge(W^A).  
\end{equation}
This is the same result as referred to in Remark \ref{rem:quasi-inverse}, in which
$\mathbb{G}$ was supposed to be algebraic.

\begin{prop}\label{prop:smooth_Hopf}
A Hopf superalgebra $A$ is smooth over $K$ if and only if the Hopf algebra 
$\overline{A}$ is smooth over $K$.
The equivalent conditions are satisfied, 
if the characteristic $\mathrm{char}\, K$ of $K$ is zero,
and if $\overline{A}$ is finitely generated. 
\end{prop}    
\begin{proof}
We refer to the proof of Theorem \ref{thm:smooth} above. 

``Only if."\ This follows easily, as was seen in the proof of (a) $\Rightarrow$ (b).  

``If."\ This follows if one argues as proving (d) $\Rightarrow$ (a), replacing the
isomorphism in (iii) of (d) with \eqref{eq:decomposition}.  

The last statement follows, since it is known that every finitely generated commutative Hopf algebra in characteristic zero
is smooth; see \cite[Page 239]{dg}.
\end{proof}

In view of Theorem \ref{thm:smooth}, Proposition \ref{prop:smooth_Hopf} generalizes 
Corollary 4.3 of Fioresi \cite{fioresi},
which essentially proves that a finitely generated Hopf superalgebra over the field $\mathbb{C}$ of
complex numbers is regular; see also the following subsection.

\subsection{Comparison with smoothness defined by Fioresi \cite{fioresi}}\label{subsec:comparison} 
Let $A=A_0\oplus A_1$ be a super-ring, that is, a $\mathbb{Z}_2$-graded ring. We say that $A$ 
is \emph{super-commutative} if (1)~the subring $A_0$ is central in $A$, and (2)~$a^2=0$ for all $a \in A_1$.
If $2$ is invertible in $A_0$, or in particular, 
if $A$ is a superalgebra over a field of characteristic $\ne 2$, then
the definition coincides with the usual super-commutativity that only requires $ab=-ba$ for all $a, b\in A_1$,
instead of (2).  
In what follows all super-rings are supposed to be super-commutative.
In fact, Schmitt \cite{schmitt} gave the definition of regularity for those super-rings,
more generally than was reproduced as Definition \ref{def:regular} above. 
Note that this definition in Section \ref{subsec:smooth_superalgebras} 
is valid for super-rings, since so is the construction of the map $\kappa_A$
as well as the argument on the Noetherian property.

As a crucial ingredient Schmitt introduced the notion 
of regular sequences to the super context, and proved that a Noetherian local super-ring $A$ is
regular if and only if the maximal super-ideal of $A$ is generated by the elements in a regular sequence;
see \cite[Theorem on Page 79]{schmitt}.  
The authors' article \cite{maszub2} in preparation will present 
an alternative approach to regular super-rings, and refine some of  
the following arguments. 

Recall from \cite[Page 66]{schmitt}, for example, that
a prime (resp., maximal) super-ideal $\mathfrak{p}$ of a super-ring $A$ is of the form 
$\mathfrak{p}_0\oplus A_1$, where $\mathfrak{p}_0$ is a prime (resp., maximal) ideal of the ring $A_0$. 
Therefore, $A$ is local if and only if $A_0$ is. Let $A$ be a Noetherian local super-ring
with maximal $\mathfrak{m}$. Then it follows 
from Definition \ref{def:regular} and the corresponding result in the non-super situation
that $A$ is regular if and only if the localization $A_{\mathfrak{p}_0}$ 
at every prime/maximal $\mathfrak{p}_0$ of $A_0$ is regular. 
We see from the last cited Theorem by Schmitt that $A$ 
is regular if and only if the $\mathfrak{m}$-adic completion 
$\widehat{A}$ of $A$ is regular.

Suppose that a regular local super-ring $A$ is a superalgebra over a field, or in other words,
$A_0$ includes a subfield; this is equivalent
to saying that the characteristics of the two rings $A_0$, $\overline{A}=A/(A_1)$ and of the residue 
field of $A_0$ all coincide. In this case the structure of $A$ is more restrictive. 
Let $\mathfrak{m}$ be the maximal super-ideal of $A$, and let $K=A/\mathfrak{m}\, (=A_0/\mathfrak{m}_0)$
be the residue field. 
Schmitt \cite[Proposition on Page 81]{schmitt} shows that
if $A$ is complete with respect to the $\mathfrak{m}$-adic topology, then it is
isomorphic to the \emph{formal power series superalgebra} over $K$, 
\[ K[[X_1,\dots, X_r]]\otimes_K \wedge(Y_1,\dots, Y_s), \]
which is the formal power series algebra in the (even) variables $X_1,\dots,X_r$ 
tensored with the free superalgebra on the set of the odd variables $Y_1,\dots, Y_s$.

Fioresi \cite[Definition 3.1]{fioresi} says that an algebraic super-variety $X$ over the field
$\mathbb{C}$ of complex numbers is \emph{smooth}, if for every closed point $P$ of $X$, 
the completion $\widehat{\mathcal{O}}_{X,P}$ of the local superalgebra $\mathcal{O}_{X,P}$ at $P$
is isomorphic to a formal power series superalgebra over $\mathbb{C}$. 
One now sees that the condition is equivalent
to saying that for every $P$,\ $\mathcal{O}_{X,P}$ is regular. 
This together with Theorem \ref{thm:smooth} prove the following proposition; 
it answers a question posed to
an earlier version of this paper. 

\begin{prop}\label{prop:compare_with_Fioresi}
An affine algebraic super-variety $X$ over $\mathbb{C}$ is smooth in the sense of Fioresi \cite{fioresi}
if and only if the coordinate superalgebra $\mathcal{O}_X$ is smooth over $\mathbb{C}$. 
\end{prop}

Let us consider the following natural question. 

\begin{Qu}\label{qu:regular_slpit}
Is every regular superalgebra $A$ over a field $K$ isomorphic 
to $\wedge_{\overline{A}}(I_A/I_A^2)$?
\end{Qu}

From Theorem \ref{thm:smooth} and Schmitt's \cite[Proposition on Page 81]{schmitt} cited above, 
one sees that
if $A$ is smooth over $K$ or complete regular local, then $A \simeq \wedge_{\overline{A}}(I_A/I_A^2)$.  
However, we will show that this is not true in general, constructing a counter-example.

Let $R$ be a Noetherian commutative algebra over a field $K$, and let $N=Rz$ be a
free $R$-module with basis $z$. Suppose that $E$ is a commutative algebra which fits into a 
Hochschild extension (see \cite[Page 311]{weibel})
\begin{equation}\label{eq:Hochschild}
0 \to N \to E \to R \to 0. 
\end{equation} 
Thus $N$ is a square-zero ideal of $E$ such that $E/N=R$. Let $M=Rw_1\oplus Rw_2$ be
a free $R$-module with basis $w_1, w_2$. Then
\[ A = E \oplus M \]
uniquely turns into a (Noetherian) super-commutative superalgebra over $K$ so that 
\[
A_0=E,\quad A_1=M,\quad w_1 w_2=z.
\]
In view of 
\[
I_A = N\oplus M,\quad I_A^2=N,\quad I_A^3=0, \quad \overline{A}= R,\quad I_A/I_A^2=M 
\]   
we have the following:
\begin{itemize}
\item $A$ is regular if and only if $R$ is regular.
\item $A\simeq \wedge_{\overline{A}}(I_A/I_A^2)$ 
if and only if the Hochschild extension \eqref{eq:Hochschild} splits. 
\end{itemize}

Therefore, for the desired counter-example, it suffices to construct a non-split extension of the
form \eqref{eq:Hochschild} with $R$ regular. We choose below an algebro-geometric example of 
such $R$. 

\begin{example}\label{ex:non-split}
Let $k$ be a field of characteristic $p>0$, and
let $K=k(t)$ be the rational function field over $k$.
Define
\[ R=K[x, y]/(x^2-y^p-t). \]
This $R$ is indeed regular. But it is not smooth over $K$ since 
the base extension $R\otimes K(t^{1/p})$, localized at the maximal ideal $(x,\ y-t^{1/p})$, is not
regular. Given $\alpha \in R$, one can define a Hochschild extension of $R$ by $N=Rz$, 
\[ E_{\alpha}= R \oplus N, \]
so that $x^2=y^p-t+\alpha z$ on $E_{\alpha}$. 
A direct computation shows that $\alpha \mapsto E_{\alpha}$
induces an $R$-linear isomorphism 
\[ R/R(2x) \overset{\simeq}{\longrightarrow} H^2_{s}(R, N), \]
where $H^2_{s}(R, N)$ denotes the symmetric 2nd Hochschild cohomology group 
\cite[Page 313]{weibel} which classifies the Hochschild extensions. Therefore, the extension 
$E_{\alpha}$ splits if and only if $\alpha$ is divided by $x$, when $p >2$, and
$\alpha = 0$, when $p=2$. There thus exist non-split extensions in any characteristic. 
 \end{example}
 
\section*{Acknowledgments}
The first-named author was supported by JSPS Grant-in-Aid for Scientific Research (C)~~26400035. 
The second-named author was supported by RFFI Grant 15-31-21169.

\end{document}